\documentclass[10pt]{amsart}

\usepackage{dsfont, amsmath, amsfonts, amsthm, amssymb, times}
\usepackage{subcaption}

\usepackage{cancel}

\usepackage[foot]{amsaddr}
\usepackage{graphicx}
\usepackage{mathtools}
\numberwithin{equation}{section}
\usepackage{setspace}
\usepackage[dvipsnames]{xcolor}
\usepackage[nomargin]{fixme}
\usepackage{setspace}
\usepackage[colorinlistoftodos,prependcaption,textsize=tiny]{todonotes}
\newcounter{mycomment}

\usepackage{hyperref}

\theoremstyle{plain}\newtheorem{theorem}{Theorem}[section]
\theoremstyle{plain}\newtheorem{prop}[theorem]{Proposition}
\theoremstyle{plain}\newtheorem{lemma}[theorem]{Lemma}
\theoremstyle{plain}\newtheorem{cor}[theorem]{Corollary}
\theoremstyle{definition}\newtheorem{definition}[theorem]{Definition}
\theoremstyle{remark}\newtheorem{remark}[theorem]{Remark}
\theoremstyle{assumption}

\title{A non-autonomous variational problem describing a nonlinear Timoshenko beam}
\author{D. Corona$^1$, A. Della Corte$^2$, F. Giannoni$^3$}
\address{$^{1,2,3}$: Mathematics Division,
School of Science and Technology,
University of Camerino (Italy).}

\usepackage{mathtools}
\usepackage{pgfplots}	
\pgfplotsset{compat=newest} %<------ Here
\begin{document}
\begin{abstract}
	We study the non-autonomous variational problem: 
	\begin{equation*}
		\inf_{(\phi,\theta)}
		\bigg\{\int_0^1 \bigg(\frac{k}{2}\phi'^2 +
		\frac{(\phi-\theta)^2}{2}-V(x,\theta)\bigg)\text{d}x\bigg\}
	\end{equation*}
	where $k>0$, $V$ is a bounded continuous function,
	$(\phi,\theta)\in H^1([0,1])\times L^2([0,1])$ and $\phi(0)=0$
	in the sense of traces. The peculiarity of the problem is its setting in the product of spaces of different regularity order.
	Problems with this form arise in elastostatics,
	when studying the equilibria of a nonlinear Timoshenko beam under distributed load,
	and in classical dynamics
	of coupled particles in time-depending external fields.
	We prove the existence and qualitative properties of global minimizers and study,
	under additional assumptions on $V$,
	the existence and regularity of local minimizers.
\end{abstract}

\maketitle

%------
% INSERT THE BODY OF THE PAPER HERE (except
% acknowledgments, funding info and bibliography)
%------

\section{Setting of the problem}
\noindent Let us indicate by $L^2:=L^2([0,1],\mathbb{R})$ and $H^1:=H^1([0,1],\mathbb{R})$ the usual Lebesgue and Sobolev spaces, and by $H^1_*\subset H^1$ the subspace of functions $\phi$
verifying $\phi(0)=0$ in the sense of traces. For $k$ a strictly positive constant and $V:[0,1]\times\mathbb{R}\to\mathbb{R}$ a bounded continuous function, we study the variational problem:
\begin{equation}
	\inf_{(\phi,\theta)\in \mathfrak{S}}
	\bigg\{\int_0^1 \bigg(\frac{k}{2}\phi'^2 +
	\frac{(\phi-\theta)^2}{2}-V(x,\theta)\bigg)\text{d}x\bigg\}
	\label{1}
\end{equation}
where the pair $(\phi,\theta)$ is searched in
\[
	\mathfrak{S}:=H^1_*\times L^2.
\]
In the following, we endow $\mathfrak{S}$
with the natural product metric and topology.
In particular, when we talk about a \textit{local} minimizer of a functional $F$
defined over $\mathfrak{S}$,
we mean a pair $(\widetilde\phi,\widetilde\theta)\in \mathfrak{S}$
such that $F(\widetilde\phi,\widetilde\theta)\le F(\phi,\theta)$
for every $(\phi,\theta)$ belonging to a sufficiently small open ball,
centered in $(\widetilde\phi,\widetilde\theta)$, with respect to  this product topology.
Clearly, in our terminology,
every global minimizer of a functional defined on $\mathfrak{S}$
is also a local minimizer.

\noindent The problem admits, for instance, the following physical interpretations:
\begin{enumerate}
	\item $\phi,\theta$ represent the kinematical descriptors of an inextensible, geometrically nonlinear Timoshenko beam submitted to a distributed load depending on $V$; 
	\item $\phi,\theta$ represent the Lagrangian coordinates of two bodies $B_\phi$, $B_\theta$ having quadratic attractive interaction potential and $x$ represents time; $B_\phi$ has mass $k$ while $B_\theta$ has negligible mass but is sensitive to an external time-dependent (electric or magnetic) field depending on $V$.
\end{enumerate}

\noindent In the following we mainly refer to the first interpretation.
Let us therefore recall that a Timoshenko beam is a one-dimensional elastic body whose kinematics is described by a curvilinear parametrization $\boldsymbol{\chi} :[0,1]\to \mathbb{R}^2$ and an extra kinematical variable $\phi$ interpreted as the \textit{orientation} of the cross-section, hence it is the angle between the cross-section of the beam and a reference axis.
A schematic representation of a Timoshenko beam is shown in Fig.\ref{scheme}.

\noindent When the material behavior is assumed to be linear and the model is
also \textit{geometrically} linearized,
we get the original formulation of the
Timoshenko beam elastic energy functional (see \cite{_Timo,Timo}), namely
\begin{equation}
	\int_0^1 \bigg(\frac{k}{2}(\phi')^2 + \frac{(\phi-\chi_2')^2}{2} \bigg)\text{d}x,
	\label{linear_}
\end{equation}
where $\chi_2$ is the vertical component of $\boldsymbol{\chi}$.
If geometric nonlinearities are considered, the elastic energy of the beam reads as follows (for a detailed derivation, see \cite{Battista}): 
\begin{equation}
	\int_0^1 \bigg(\frac{k}{2}\phi'^2 + \frac{(\phi-\theta)^2}{2} \bigg)\text{d}x,
	\label{nonlinear_}
\end{equation}
where it was assumed that the beam has length 1 and is inextensible, that is 
\begin{equation}
	||\boldsymbol{\chi}'||^2\equiv 1.
	\label{inextensible}
\end{equation}
The bending coefficient $k$ belongs to $\mathbb{R}^+$ and the function $\theta$ verifies $\boldsymbol{\chi}'=(\cos\theta,\sin\theta)$.

\begin{figure}
	\includegraphics[scale=0.2]{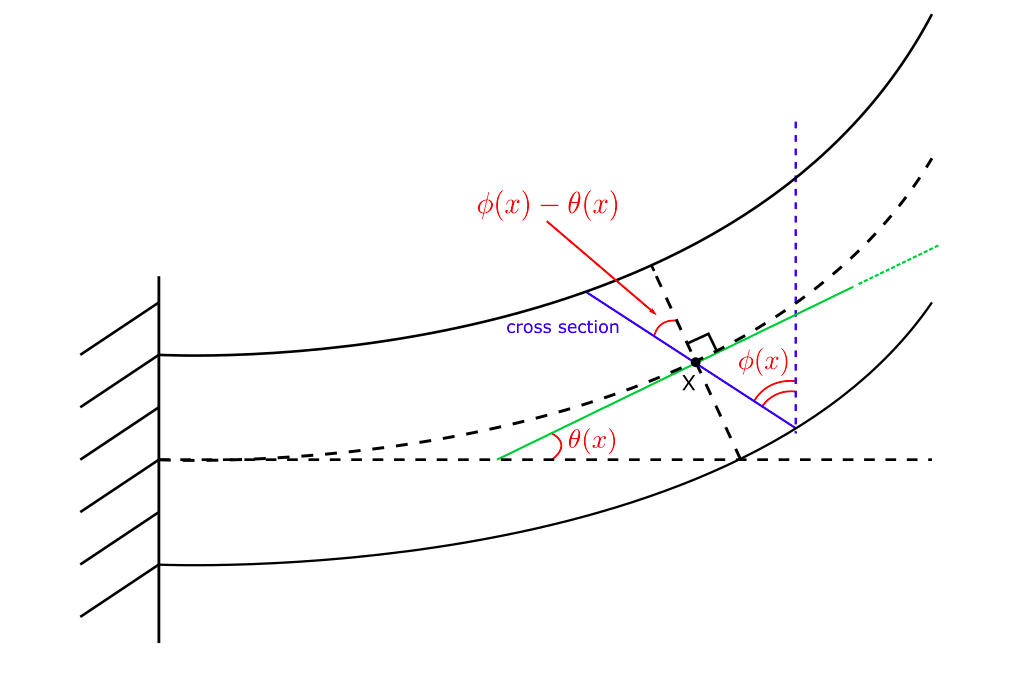} 
	\caption{A schematic representation of a Timoshenko beam.
	Notice that the inextensibility assumption \eqref{inextensible} makes it possible to interpret the variable $x$ as a curvilinear abscissa along the middle line (the dashed curve in the picture).
}
	\label{scheme}
\end{figure}

\noindent The potential due to a distributed load
$\boldsymbol{b} (x)$ is 
$
\int_0^1 \big( {\boldsymbol{b} (x)} \cdot \boldsymbol{\chi}(x) \big)\text{d}x
$
which can be rewritten, using an integration by parts,
as $\int_0^1 \big(\boldsymbol{B} (x) \cdot \boldsymbol{\chi}'(x) \big)\text{d}x$,
where $\boldsymbol{B}(x)=\int_x^1 \boldsymbol{b}(\xi)\text{d}\xi$
and the kinematical constraint
\begin{equation}\label{constraintchi}
	\boldsymbol{\chi}(0)=0
\end{equation}
was imposed.
Notice that \eqref{nonlinear_} reduces to \eqref{linear_}
under the smallness assumption $\chi_2' =\sin\theta \approx \theta$.
The minimization of the total energy, if the load is uniform and has zero horizontal component, is therefore of the form \eqref{1} with \begin{equation}\label{sin_}
	V(x,\theta)= b(1-x)  \sin\theta,
\end{equation} 
where  $b>0$ corresponds to the load density per unit length.
Adding to the constraint \eqref{constraintchi} the further requirement 
\begin{equation}
	\label{constraint_phi}
	\phi(0)=0,
\end{equation}
we obtain the conditions usually expressed saying that
the beam is horizontally clamped at one of its extremes.
Notice that, in the variational formulation \eqref{1},
the constraint \eqref{constraint_phi} is contained in the definition of $\mathfrak{S}$,
while the constraint \eqref{constraintchi} has to be taken into account when
reconstructing the vector field $\boldsymbol{\chi}(x)$ from $\theta(x)$.
The case in which $V$ does not depend on $x$ was studied in \cite{Battista},
while a numerical investigation of the case with distributed load
was performed in \cite{Dellisola}.

\noindent The variational problem \eqref{1} is close to a model-case of non-autonomous,
not strictly convex problem.
The absence of a term with $\theta'^2$ makes the
integrand not (strictly) convex in the highest order derivative
(so that convergence of minimizing sequences is not granted by standard arguments)
and at the same time settles the problem in the ``asymmetric" space $H^1\times L^2$.
When $k=1$, an effective way to see the above mentioned asymmetry of the
problem is writing it follows:
\begin{equation*}
	\inf_{\phi,\theta}
	\bigg\lbrace
		\frac 12\bigg(\Vert\phi\Vert_{H^1}^2+\Vert\theta\Vert_{L^2}^2\bigg)
		-(\phi,\theta)_{L^2}-\int_0^1 V(x,\theta)\text{d}x
	\bigg\rbrace.
\end{equation*}
When $k\ne1$, an analogous representation can be obtained considering
an equivalent metric on $H^1$.

\noindent The results developed herein all hold when $V$ is given by \eqref{sin_}. However,
we will not limit ourselves to this form of $V$
for the existence and the main properties of the global minimizer.

\noindent Specifically, in Section \ref{sec:global-min}
the existence of a global minimizer of problem \eqref{1} will be
proved assuming that $V$ is a bounded continuous function.
Some general properties verified by global minimizers will be
established under the further assumption that $\theta\to V(x,\theta)$ admits a
global maximum at $a>0$ (independently of $x$) and that
$V(x,-\theta)<V(x,\theta)$ for $\theta\in [0,a]$ and for almost every $x\in [0,1]$.
In Section \ref{sec:local_minimizers},
the existence and regularity of local
minimizers, different from the global ones,
will be studied under the assumption that $V$ is given by \eqref{sin_},
and that both $b$ and $k/b$ are sufficiently small.
The main motivation of this study
is the existence of minimizers with similar properties in the simpler cases of Euler
beam under distributed load (see \cite{Euler_,Euler__}) and of Timoshenko beam
under concentrated end load (see \cite{Battista}). \\

\section{Global minimizers}
\label{sec:global-min}

Let us assume that function $V$ in \eqref{1} is a bounded
continuous function. We remark that, since we do not have a term depending on
$\theta'$ in the integral \eqref{1}, the fact that a minimizing sequence
$(\phi_n,\theta_n)$ has a bounded energy does not provide any information for
the derivatives of $\theta_n$.
Hence, the weak convergence in $L^2$ of $\theta_n$ to a
function $\theta$ does not imply that $\theta$ minimizes the energy, so the
usual direct method of the calculus of variations must be used with
caution.

\noindent Let us reformulate problem \eqref{1} as
\begin{equation}
	\inf_{(\phi,\theta)\in \mathfrak{S}} F(\phi,\theta),
	\label{F___}
\end{equation}
with 
\begin{equation}
	F(\phi,\theta)=
	\int_0^1 \bigg(\frac{k}{2}\phi'^2 + \frac{(\phi-\theta)^2}{2}-V(x,\theta) \bigg)\,
	\text{d}x.
	\label{functionalF}
\end{equation}
We establish the existence of a minimizer of $F(\phi,\theta)$
in the following proposition.

\begin{prop}
	\label{th1}
	Assume that $k>0$ and that $V$ is a bounded continuous function on
	$[0,1]\times \mathbb R$.
	Then Problem \eqref{F___} admits a solution.
\end{prop}

\begin{proof}
	\medskip

	Fix $M > |V|$. Then $F(0,0)\leq M$ and the infimum can be searched among functions $(\phi,\theta)$ satisfying $F(\phi,\theta)\leq M$ and
	such functions satisfy
	\[
		\int_0^1 \frac k2 \phi'^2\text{d}x\le 2M.
	\]
	The infimum can therefore be searched assuming that $\|\phi\|_{H^1_*}\le C$ for some constant $C$ depending only on $M$ and $k$, so that it follows that $\|\phi\|_{C^0([0,1])}\le \|\phi\|_{H^1_*} \le C$.

	\medskip 
	\noindent Let us define the function $H\colon [0,1]\times \mathbb{R}^2 \to \mathbb{R}$
	as
	\[
		H(x,\phi,\theta):= -\phi\theta+\frac{\theta^2}{2}-V(x,\theta),
	\]
	so that we have
	\[
		F(\phi,\theta) = 
		\int_0^1 \bigg(\frac{k}{2}\phi'^2 + \frac{\phi^2}{2}+ H(x,\phi,\theta) \bigg)\,
		\text{d}x.
	\]

	\noindent Let us now set, for any $x\in[0,1]$ and $\phi\in[-C,C]$, 
	\begin{equation}
		\label{K}
		K(x,\phi):=\inf_{\theta\in \mathbb{R}}(H(x,\phi,\theta)).
	\end{equation}
	It is easy to check that any real number $\theta$ satisfying $H(x,\phi,\theta)\leq H(x,\phi,0)\le M$,
	satisfies $|\theta |\le D:=C+\sqrt{2M+C^2}$.
	Therefore the set of solutions of problem \eqref{K}
	is a non-empty closed subset of the compact $[-D,D]$.
	For every $x\in[0,1]$, let us indicate by $\theta_\phi(x)$ the smallest solution of  \eqref{K}.
	Being lower-semicontinuous, $\theta_\phi$ is a measurable function, and being bounded it is in $L^2([0,1])$.

	\noindent Let $(x_n,\phi_n)$ be a sequence converging to $(x,\phi)$.
	Up to a subsequence,
	$\theta_{\phi_n}$ converges to some $\theta_\phi \in [-D,D]$
	and we have
	\[
		H(x_n,\phi_n,\theta_{\phi_n}) = K(x_n,\phi_n) \leq H(x_n,\phi_n,\theta_\phi).
	\]
	As $H$ is continuous, passing to the limit we get 
	\[
		K(x,\phi)\leq H(x,\phi,\theta_\phi)
		= \lim K(x_n,\phi_n) \leq  H(x,\phi,\theta_\phi) =K(x,\phi).
	\]
	Hence $K$ is a bounded continuous function.
	By standard arguments of the calculus of variations
	(see for instance Tonelli's existence theorem in \cite{buttazzo}),
	the problem 
	\begin{equation}
		\inf_\phi \int_0^1 \bigg(\frac{k}{2}\phi'^2
		+ \frac{\phi^2}{2}+K(\phi,x)\bigg)\text{d}x
		\label{proplem_phi}
	\end{equation}
	admits a solution $\bar \phi$ in $H^1_*([0,1])$
	and
	\[
		\inf_{(\phi,\theta)}F(\phi,\theta)
		\leq F(\bar\phi, \theta_{\bar \phi})
		= \inf_\phi \int_0^1 \bigg(\frac{k}{2}\phi'^2
		+ \frac{\phi^2}{2}+K(\phi,x)\bigg)\text{d}x \leq \inf_{(\phi,\theta)}F(\phi,\theta).
	\]
	Hence, 
	the pair $(\bar\phi, \theta_{\bar \phi})$ is a solution to Problem \eqref{1}. 
\end{proof}

%%%%%%%%%%%%%%%%%%%%%%%%%%%%%%%%%%%%%%%%%%%%%%%%%%%%%%%%%%%%%%%

%\subsection{Some properties of the minimizers}
%\label{sec:glob_min_properties}
% In this section
\begin{remark}
	In principle, the minimizer whose existence has been established in Proposition \ref{th1} may fail very badly to be unique,
	even whether the related Problem \eqref{proplem_phi},
	which is a classical problem of calculus of variations,
	admits a unique solution $\bar{\phi}$. 

	\noindent Suppose, for instance, that $V(x,\theta)$ has the following form:
	\begin{equation*}
		V(x,\theta)=g\big(\theta-f(x)\big)
	\end{equation*}
	where $g$ is such that
	$\frac {\partial^2 H}{\partial \theta^2}=1
	-g''(\theta-f(x))> 0$ when $(\theta-f)$ belongs to some open set $I$.
	Then every $\xi$ such that $$\xi-\phi-g'(\xi-f(x))=0$$ is
	a solution of problem \eqref{K} if $\xi-f\in I$.
	
	\noindent Suppose now that $\bar{\phi}|_A\equiv f|_A$ on a set $A\subset[0,1]$ of positive
	measure, and that the
	equation $s=g'(s)$ has two distinct solutions $s_1,s_2\in I$
	such that, for $\theta_1(x)=s_1+\phi(x)$
	and $\theta_2(x)=s_2+\phi(x)$, we have
	\[
		H(x,\phi,\theta_1)=H(x,\phi,\theta_2).
	\]
	
\noindent Then the problem \eqref{K} is solved by
	both $\theta_1$ and $\theta_2$.
	In these hypotheses,
	$(\bar{\phi},\theta^*)$ is a minimizer of $F$ for every
	$\theta^*$ defined as follows:
	\[
		\theta^*=\theta_1\ \text{for} \ x\in B\ ,
		\ \ \theta^*=\theta_2\  \text{for}\  x\in A\setminus B \ \ ,
		\ \ \theta^*=\theta_{\bar{\phi}}\  \text{for}\  x\in [0,1]\setminus A
	\]
	where $B\subset A$ is a (completely arbitrary) subset of positive measure.

	\noindent Pathological phenomena of this type are well known
	(similar problems were already discussed, for instance,
	in the classical works \cite{Young1937,Young1938}),
	and are usually addressed by means of relaxation theory
	(see e.g. \cite{Dacorogna}, Chapter III),
	which however has not been developed, to the best of our knowledge,
	for problems living in the product of Sobolev
	spaces of different regularity order.
	In the following, we will be mainly concerned with cases in which $V(x,\theta)$
	does not produce such pathological multiplicity of minimizers.

\end{remark}

\vspace{0.5cm}

\noindent We shall prove now some properties
of the global minimizers of Problem \eqref{1}.
In addition to the information they provide on the problem,
these results will ensure that the local minimizers
studied in Section \ref{sec:local_minimizers}
are necessarily not \textit{global} minimizers.

% 
%Let us thus establish some properties of any minimizer
%$(\bar\phi, \bar\theta)$ of \eqref{1}, holding when $V$ is given by \eqref{sin_}.
%We first show that  $\bar\phi$ and $\bar\theta$ remain in the interval $[0,\frac\pi 2]$.

\begin{lemma}\label{2}
	Assume, in addition to the assumptions of Proposition \ref{th1},
	that there exists $a>0$ such that 
	for almost every $x\in[0,1]$, for every $\theta\in\mathbb{R}$,
	$V(x,\theta)\leq V(x,a)$ and for every $\theta\in (0,a]$, $V(x,-\theta)< V(x,\theta)$.
	Then any minimizer $(\widetilde \phi,\widetilde \theta)$  of \eqref{1} takes values in $[0,a]\times[0,a]$.
\end{lemma}
\begin{proof}
	Define $\widetilde V$ by $\widetilde V(x,\theta)=V(x,\theta)$ if $\theta<a$, $\widetilde
	V(x,\theta)=V(x,a)$  if $\theta\geq a$, so that $\widetilde V$ now satisfies,
	for almost every $x\in[0,1]$ and for every $\theta\in\mathbb{R}$,
	$\widetilde{V}(x,\theta)\leq \widetilde{V}(x,a)$ and $\widetilde{V}(x,\theta)\leq \widetilde{V}(x,|\theta|)$.
	We set
	\begin{equation}
		\widetilde F(\phi,\theta):=\int_0^1 \bigg(\frac{k}{2}\phi'^2 + \frac{(\phi-\theta)^2}{2}-\widetilde V(x,\theta)\bigg)\text{d}x.
		\label{Ftilde}
	\end{equation}
	Clearly $\widetilde F \leq F$. Moreover, we have that 
	\[
		\widetilde F (|\widetilde\phi|,|\widetilde\theta|)
		\leq \widetilde F(\widetilde\phi,\widetilde\theta)
		\quad \text{and}\quad
		\widetilde F (\min(|\widetilde\phi|,a),\min(|\widetilde\theta|,a))
		\leq \widetilde F (|\widetilde\phi|,|\widetilde\theta|)
	\]
	as all integrands in \eqref{Ftilde} do not increase with these replacements. Set 
	\[
		(\bar\phi,\bar\theta):=(\min(|\widetilde\phi|,a),\min(|\widetilde\theta|,a)).
	\]
	As these functions take values in $[0,a]$, we have $\widetilde F (\bar\phi,\bar\theta)= F (\bar\phi,\bar\theta)$.
	Hence
	\[
		F (\bar\phi,\bar\theta)\leq \widetilde F(\widetilde\phi,\widetilde\theta)
		\leq F(\widetilde\phi,\widetilde\theta)\leq F (\bar\phi,\bar\theta).
	\]
	This implies that the previous inequalities were in fact equalities. Since
	$\int_0^1 \frac{k}{2}\phi'^2\text{d}x$ does not decrease when replacing
	$|\widetilde\phi|$ by $\min\{|\widetilde\phi|,a\}$, it follows that $|\widetilde\phi(x)|\leq
	a$ almost everywhere (and thus everywhere as $\widetilde\phi$ is
	continuous).
	Moreover, since $\int_0^1 \frac{(\phi-\theta)^2}{2}\text{d}x$
	does not decrease when replacing  $|\widetilde\phi|$ and $|\widetilde\theta|$  by
	$\min\{|\widetilde\phi|,a\}$ and $\min\{|\widetilde\theta|,a\}$, it follows that the sets
	$\{x,\ |\widetilde\phi(x)|>a\}$ and  $\{x,\ |\widetilde\theta(x)|>a\}$ coincide up
	to a null set.
	Thus $|\widetilde\theta(x)|\leq a$ almost everywhere.
	Finally, noting that for every $\theta\in (0,a]$, we have $V(x,-\theta)<
	V(x,\theta)$ a.e. on $[0,1]$, and that $\int_0^1 \widetilde V(x,\theta)
	\text{d}x$ does not decrease when replacing  $\widetilde\theta$ by
	$|\widetilde\theta|$, it follows that the set $\{x\in[0,1]:\ \widetilde\theta(x)\in [-a,0)
	\}$ has null measure, hence  $\widetilde\theta(x)\geq 0$ almost everywhere. 
\end{proof}

\begin{prop}
	\label{prop:regularity_global_min}
	In addition to the assumptions of Lemma \ref{2},
	assume that, for every $x\in[0,1]$,
	the function $\theta\to V(x,\theta)$ is of class $C^k(\mathbb{R})$,
	with $k \ge 2$,
	and satisfies $\frac{\partial^2 V}{\partial \theta^2}(x,\theta)\ne 1$
	in $[0,1]\times [0,a]$.
	If $(\widetilde{\phi},\widetilde{\theta})$
	is a minimizer of \eqref{1},
	then $\widetilde\theta \in C^{k-1}(\mathbb{R})$
	and $\widetilde\phi \in C^{k + 1}(\mathbb{R})$.
	Moreover,
	if $V$ is a $C^\infty$ ($C^\omega$) function,
	then both $\widetilde{\phi}$ and $\widetilde{\theta}$ are
	$C^\infty$ ($C^\omega$) functions.
\end{prop}

\begin{proof}
	For every $x\in [0,1]$, we have 
	\begin{equation*}
		\label{condition_theta}
		\widetilde{\theta}=
		\min_{\theta}
		\bigg(-\widetilde{\phi}(x)\theta + \frac{\theta^2} {2}-V(x,\theta)\bigg)
	\end{equation*} 
	so $\widetilde{\theta}(x)$ has to solve 
	\begin{equation}
		\label{_theta_}
		-\widetilde{\phi}(x)+
		\widetilde\theta-\frac{\partial V}{\partial \theta}(x,\widetilde\theta)=0
	\end{equation}
	Let us set $f(x,\theta):=\theta - \frac {\partial V}{\partial \theta}(x, \theta)$.
	By the hypotheses of Lemma \ref{2},
	for every $x\in [0,1]$ we have
	\[
		\frac{\partial V}{\partial \theta}(x,0) > 0
		\quad\text{and}\quad
		\frac{\partial V}{\partial \theta}(x,a) = 0,
	\]
	thus
	\[
		f(x,0) < 0
		\quad\text{and}\quad
		f(x,a) = a.
	\]

	\noindent By Lemma \ref{2},  $\widetilde\theta$
	takes values in $[0,a]$,
	so, by hypothesis we have
	\begin{equation}
		\label{deriv_}
		\frac{\partial f}{\partial \theta}(x,\theta)
		= 1-\frac{\partial^2 V}{\partial \theta^2}(x,\theta) \ne 0.
	\end{equation}
	As a consequence,
	$f$ is strictly increasing with respect to $\theta$
	in $[0,1]\times [0,a]$,
	and 
	for every $x\in[0,1]$
	there exists a unique value of $\widetilde\theta \in [0,a]$
	such that \eqref{_theta_} holds.

	\noindent By the inverse function theorem, there exists a function
	$g\colon [0,a]\times [0,1] \to [0,a]$,
	with the same regularity of $f$, hence of class $C^{k-1}$,
	such that 
	\begin{equation}
		\label{eq:glob_theta_g}
		\widetilde\theta(x) = g(x,\widetilde\phi(x)).
	\end{equation}
	As a consequence, $\widetilde\theta \in C^0([0,1])$.
	Since $(\widetilde\phi,\widetilde\theta)$ is a minimizer, we have
	\[
		\text{d}F(\widetilde\phi,\widetilde\theta)[\xi,0]
		= \int_{0}^{1} \bigg(
			k \widetilde\phi' \xi' + (\widetilde\phi - \widetilde\theta)\xi
		\bigg)\text{d}x = 0,
		\quad \forall \xi \in C^\infty_c([0,1]),
	\]
	hence
	\begin{equation}
		\label{eq:variation}
		k\widetilde\phi' \xi' + (\widetilde\phi- \widetilde\theta)\xi = 0.
	\end{equation}
	Since $\widetilde\theta$ is continuous, 
	\eqref{eq:variation} implies that $\widetilde\phi$ is $C^1$.
	Then, using again \eqref{eq:glob_theta_g},
	we obtain that $\widetilde\theta$ is of class $C^1$
	and by \eqref{eq:variation} we get that $\widetilde\phi$ if of class $C^2$.
	Hence, by a standard argument, we obtain 
	\begin{equation}
		\label{eq:glob_phi}
		k\widetilde\phi'' = \widetilde\phi - \widetilde\theta,
	\end{equation}
	and iterating \eqref{eq:glob_theta_g} and \eqref{eq:glob_phi}
	we obtain the desired regularity.

	\noindent If $V$ is of class $C^{\infty}$,
	then by induction both $\widetilde\phi$ and $\widetilde\theta$
	are of class $C^\infty$.

	\noindent Assume finally that $V$ is real-analytic.
	Applying the real-analytic version of the inverse function theorem
	(see e.g. \cite{analyticreal}, p. 47)
	to \eqref{_theta_},
	we can replace $\theta$ in \eqref{eq:variation} by an analytic
	function $G(\phi,x)$ to obtain the boundary value problem:
	\begin{equation*}
		\begin{cases}
			-k\phi''+\phi-G(\phi,x)=0,\\
			\phi(0)=0,\\
			\phi'(1)=0,
		\end{cases}
		\label{bvp}
	\end{equation*}
	which will be solved pointwise by $\widetilde{\phi}$.
	By Cauchy-Kovalevskaya theorem
	(for an ODE version, which is in fact a particular case,
	see for instance \cite{teschl}, theorem 4.1)
	we obtain $\widetilde\phi\in C^\omega([0,1])$ as well,
	whence $\widetilde\theta\in C^\omega([0,1])$ too.
\end{proof}

\noindent From the previous proof, it is clear that the regularity of 
local minimizers of $F$ depends on the possibility to invert the function $f$,
which is ensured if \eqref{deriv_} holds.
As a consequence, we have the following result.
\begin{cor}
	\label{cor:regul_loc_min}
	Let $V\colon [0,1]\times\mathbb{R} \to \mathbb{R}$
	be a bounded function of class $C^k$, with $k \ge 2$,
	such that 
	\[
		\frac{\partial^2 V}{\partial \theta^2}(x,\theta) \ne 1,
		\quad \forall(x,\theta)\in [0,1]\times \mathbb{R}.
	\]
	If $(\widetilde\phi,\widetilde\theta)$ is a local minimizer of $F$,
	then $\widetilde\theta \in C^{k-1}(\mathbb{R})$
	and $\widetilde\phi \in C^{k + 1}(\mathbb{R})$.
	Moreover,
	if $V$ is a $C^\infty$ ($C^\omega$) function,
	then both $\widetilde{\phi}$ and $\widetilde{\theta}$ are
	$C^\infty$ ($C^\omega$) functions.
\end{cor}

\noindent Lemma \ref{2} and Proposition \ref{prop:regularity_global_min}
apply to the problem \eqref{1} with
\[
	V(x,\theta)=b(1-x)\sin\theta
	\quad\text{and}\quad
	a=\frac{\pi}{2}.
\]
As a consequence, we can give the following result.
\begin{lemma}
	\label{lem:monotonicity}
	If $(\widetilde{\phi},\widetilde{\theta})$ solves problem \eqref{F___}-\eqref{functionalF}
	with $V(x,\theta)=b(1-x)\sin\theta$,
	then $\widetilde{\phi}$ is strictly increasing.
\end{lemma}
\begin{proof}

	Since for every $(x,\theta) \in [0,1]\times[0,\pi/2]$ we have
	\[
		\frac{\partial^2 V}{\partial \theta^2}(x,\theta) 
		= - b(1-x)\sin\theta \le 0,
	\]
	we can apply Proposition \ref{prop:regularity_global_min}
	and $\widetilde\phi$ is an analytic function.
	Therefore, it is piecewise monotonic and
	cannot be constant on an open interval without being constant
	(and thus equal to zero, as $\widetilde{\phi}(0)=0$) on $[0,1]$. 
	We therefore just have to exclude that there exist $0\le \alpha<\beta\le 1$
	such that $\widetilde{\phi}$ is strictly decreasing on $[\alpha,\beta]$.
	Suppose now the contrary.
	Then define, on the interval $[\alpha,1]$, the functions
	\[
		f(x)\coloneqq\max\{\widetilde{\phi}(\alpha),\widetilde{\phi}(x)\}
	\]
	and
	\[
		g(x)\coloneqq
		\begin{cases}
			f(x), &		\mbox{if } f(x)\ne \widetilde\phi(x), \\
			\widetilde\theta(x), &		\mbox{otherwise.} 
		\end{cases}
	\]
	It is easily seen that $f\in H^1_*$ and $g\in L^2$.
	Moreover, $0\le \widetilde{\phi}(x)\le f(x)\le \frac \pi 2$
	and $0\le \widetilde{\theta}(x)\le g(x) \le \frac \pi 2$.
	We have further that, a.e. on $[0,1]$, $f'(x)\le \widetilde{\phi}'(x)$
	and that $|f-g|\le |\widetilde{\phi}-\widetilde{\theta}|$.
	Since $\sin(\cdot)$ is strictly increasing in $[0,\frac \pi 2]$,
	we also have $\sin{\widetilde{\theta}}\le \sin{g}$.
	It follows that replacing $(\widetilde{\phi},\widetilde{\theta})$ by $(f,g)$ on $[\alpha,1]$
	the value of the functional \eqref{functionalF} decreases,
	which is absurd.
\end{proof}

\noindent Fig.~\ref{fig:global-minimizer-Tim} shows a numerical computation of
the global minimizer of $F$ with $V = b(1-x)\sin\theta$,
with $b = 1$ and $k = 0.01$.
It can be seen that
$\widetilde\phi$ is strictly increasing,
as it is ensured by Lemma \ref{lem:monotonicity}.

\begin{figure}[h]
	\centering
	\includegraphics[width = 0.7\textwidth]{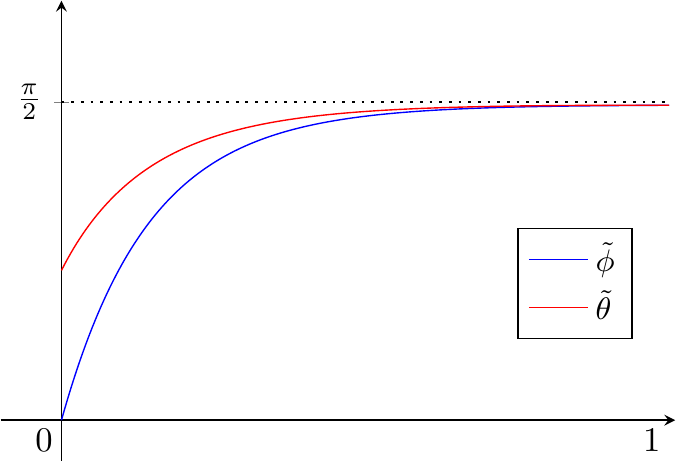}
	\caption{
		A numerically evaluated solution of \eqref{1}
		with $V(x,\theta) = b(1-x)\sin\theta$,
		$b=1$ and $k=0.01$.
	}
	\label{fig:global-minimizer-Tim}
\end{figure}

\noindent We end this section with the following result,
which stems from the proof of 
Proposition \ref{prop:regularity_global_min}
and it will be useful for the study of local minimizers 
different from the global one.
\begin{cor}
	\label{cor:Theta}
	Let $V(x,\theta) = b(1-x)\sin\theta$ and $b <1$.
	Then there exists a unique map
	$\Theta\colon H^1_* \to L^2$
	such that 
	for all $(\phi,\theta) \in H^1_*\times L^2$ we have
	\[
		F(\phi,\Theta(\phi)) \le F(\phi,\theta).
	\]
\end{cor}
\begin{proof}
	As in the proof of Proposition \ref{prop:regularity_global_min},
	let us set
	\[
		f(x,\theta) =\theta - \frac{\partial V}{\partial \theta}(x,\theta) =\theta + b(1-x)\cos\theta.
	\]
	Under the hypothesis $b < 1$, the function $f$ is strictly increasing.
	By the inverse function theorem,
	there exists a unique analytic function $g\colon [0,1]\times\mathbb{R} \to \mathbb{R}$
	such that
	\[
		f(x,g(x,\phi)) - \phi = 0
	\]
	for every $x \in [0,1]$ and $\phi \in \mathbb{R}$.
	Hence, we define the map $\Theta\colon H^1_* \to L^2$ as follows:
	\begin{equation*}
		(\Theta(\phi))(x) \coloneqq g(x,\phi(x)),
	\end{equation*}
	where we notice that, being $\phi \in H^1_* \subset C^0$,
	$\Theta(\phi)$ is continuous, hence belongs to $L^2$.
\end{proof}

\section{Local minimizers}
\label{sec:local_minimizers}

The study of local minimizers in elastostatics is typically
not easy,
and fully general methods for establishing the existence of local
minimizers which are not global ones,
as famously asked by J.M. Ball in Problem 9 of
\cite{ball}, have not yet been found.
In this section, we address the existence
of  local minimizers (different from the global one)
of a particular case of the functional defined in \eqref{Ftilde}, which we will indicate by $F_{b,k}\colon \mathfrak{S} \to \mathbb{R}$, defined as
\begin{equation}
	\label{eq:functionalF_sin}
	F_{b,k}(\phi,\theta)=
	\int_0^1 \bigg(\frac{k}{2}\phi'^2 + \frac{(\phi-\theta)^2}{2}
	-b(1-x)\sin\theta \bigg)\,
	\text{d}x.
\end{equation}
In particular, our main result is Theorem \ref{teo:local-min},
which ensures the existence of a local minimizer $(\widetilde\phi,\widetilde\theta)$
such that $\widetilde\phi(x) < 0$ for all $x \in (0,1]$.
We therefore extend the results of \cite{Battista,Euler_,Euler__},
where similar local minimizers where found for
nonlinear Euler beams under distributed load and
nonlinear Timoshenko beams under concentrated end-load (which leads to an autonomous variational problem).
\begin{theorem}
	\label{teo:local-min}
	Let $F_{b,k}\colon\mathfrak{S}\to \mathbb{R}$ be as in \eqref{eq:functionalF_sin}.
	If both $b$ and $k/b$ are sufficiently small,
	then
	there exists a local minimizer $(\widetilde\phi,\widetilde\theta)$ of $F_{b,k}$ such that
	\[
		\widetilde\phi(x) < 0,
		\qquad \forall x \in (0,1].
	\]
\end{theorem}

\noindent A numerically evaluated local minimizer when $b=1$ and $k=0.01$ is shown in Fig.~\ref{figg__}.
\begin{figure}[h]
	\includegraphics[scale=1]{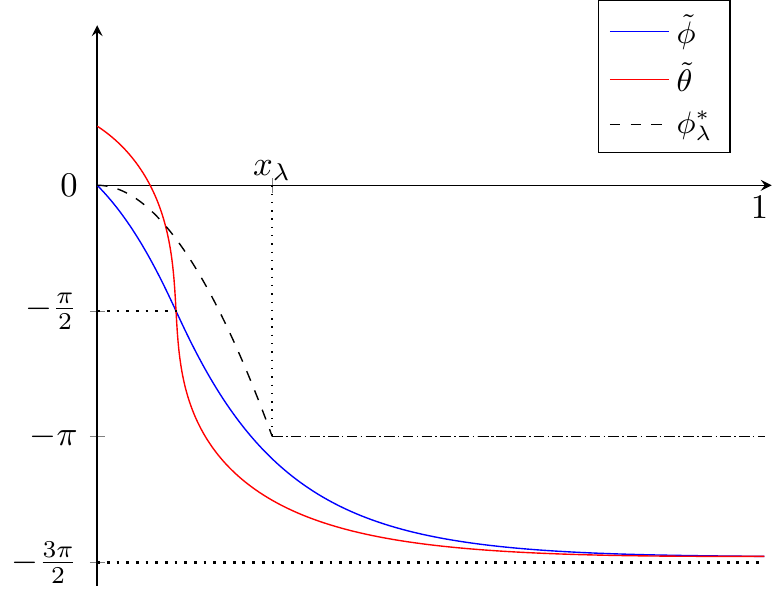}
	\caption{
		A numerically evaluated local minimizer of \eqref{eq:functionalF_sin}
		with $b=1$ and $k=0.01$.
	}
	\label{figg__}
\end{figure}

\noindent From the statement of Theorem \ref{teo:local-min} it 
is immediately clear that the ratio $k/b$ plays
a central role in the existence of local minimizers different
from the global one.
For the sake of presentation, for every fixed $b,k >0$ we
will indicate by $\lambda$ the inverse of that ratio, hence
\[
	\lambda = \frac{b}{k},
\]
so we will prove some of the following results 
provided that $\lambda$ is sufficiently large.

\noindent A key ingredient for our proof is the function $\phi^*_\lambda$,
which is given by the following definition.
Indeed,
as we are going to see during the different steps of the proof,
it provides a ``natural'' upper bound for the 
component $\phi$ of the local minimizers,
as it can be noted in Fig.~\ref{figg__}.
\begin{definition}
	For every $b,k > 0$ 
	we denote by $\lambda$ the ratio $b/k$ and
	we define the function
	$\phi_{\lambda}\colon [0,1] \to \mathbb{R}$ as
	\begin{equation}
		\label{eq:def_phiStar}
		\phi_{\lambda}^*(x)=
		\max\bigg\{
			\frac{\lambda}{2}x^2 \left(\frac{x}{3} -1\right)
			- \frac{1}{2}x^2,
			- \pi
		\bigg\}.
	\end{equation}
	We denote by $x_{\lambda}$ the least $x$ such that $\phi^* = -\pi$,
	that is:
	\[
		x_{\lambda}=
		\min\bigg\{x \in [0,1]: \phi^*_{\lambda}(x) = -\pi\bigg\},
	\]
	which is well defined if $\lambda$ is sufficiently large.
	Moreover, we set
	\begin{equation}
		\label{eq:def-C*lambda}
		\mathfrak{C}_{\lambda}^* \coloneqq
		\bigg\{
			\phi \in H^{1}_{*}: \phi(x) \le \phi_{\lambda}^*(x),
			\ \forall x \in [0,1]
		\bigg\}
	\end{equation}
	and
	\[
		\mathfrak{S}_{\lambda}^* \coloneqq
		\mathfrak{C}_{\lambda,\epsilon}^* \times 
		L^2 \subset \mathfrak{S}.
	\]
	% Finally, we will write simply $\phi^*_\lambda$, $x_\lambda$, $\mathfrak{S}_{\lambda}^*$ and $\mathfrak{C}^*_{\lambda}$ meaning respectively $\phi^*_{\lambda,0}$, $x_{\lambda,0}$, $\mathfrak{S}_{\lambda,0}^*$ and $\mathfrak{C}^*_{\lambda,0}$.
\end{definition}

\noindent The main idea of the proof of Theorem \ref{teo:local-min}
is showing that the global minimizer of $F_{b,k}$ in $\mathfrak{S}^*$,
denoted by $(\widetilde\phi,\widetilde\theta)$, is strictly less than $\phi^*_\lambda$,
except in $0$.
The special form of $\phi^*_\lambda$
implies that $\widetilde\phi$ can ``touch" it only at $x_\lambda$,
and this is proved in Subsection \ref{sub:31}.
Subsection \ref{sub:32} is devoted to prove
that $\widetilde\phi(x_\lambda)$
is actually also strictly less then $\phi^*_\lambda(x_\lambda) = -\pi$.
As a first step
we show that if $b$ is sufficiently small
and $b/k=\lambda$ remains constant,
then $\widetilde\phi$ is arbitrarily close,
with respect to the $C^1$ norm,
to the minimizer of the Euler beam problem, hence to
the global minimizer of 
\[
	\phi \mapsto
	\int_{0}^1\bigg(\frac{\phi'^2}{2} - \lambda(1-x)\sin\phi\bigg)\text{d}x,
\]
subject to $\phi(0) = 0$ and $\phi(x)\le \phi^*_\lambda(x)$. 
As a second step,
we prove that if $\lambda$ is sufficiently large,
and thus if $k/b$ is sufficiently small,
then such a minimizer is strictly less then $-\pi$ at $x_\lambda$.
In Subsection \ref{sub:33} we formally 
give the proof of Theorem \ref{teo:local-min},
recollecting all the previous results
and using a $\Gamma$--convergence argument
to show that $(\tilde\phi,\tilde\theta)$ is indeed
a local minimizer on the whole set $\mathfrak{S}$.

\subsection{General results for minimizers in $\mathfrak{S}_{\lambda}^*$}
\label{sub:31}
In this section we provide some results
that hold for the minimizers of $F_{b,k}$ in $\mathfrak{S}_{\lambda}^*$,
independently of $b,k > 0$. 
As a first step, we give the following existence result.
\begin{prop} 
	\label{prop:P*existence}
	For every $b,k > 0$,
	there exists a global minimizer of $F_{b,k}$ in $\mathfrak{S}_{\lambda}^*$.
\end{prop}
\begin{proof}
	The set $\mathfrak{C}_{\lambda}^*\subset H^1_*$
	is convex and closed with respect to the $L^{\infty}$ norm
	and therefore it is closed with respect to the weak convergence in $H^1$.
	Since the minimizing sequences weakly converge in $H^1$,
	their weak limit belongs to $\mathfrak{C}_{\lambda}^*$ (see for instance Theorem 7.3.7 in \cite{Kurdila}).
\end{proof}

\noindent Since $\mathfrak{S}_{\lambda}^*$ is a closed set with boundary,
a global minimizer does not satisfy the Euler-Lagrange equations in general.
However, the form of $\phi_{\lambda}^*$
allows us to prove that this is actually the case,
as stated by the following proposition, which is the main result of this subsection.
\begin{prop}
	\label{prop:local_EulerLagrange}
	Let $(\widetilde{\phi},\widetilde{\theta})$ be a global minimizer
	of $F_{b,k}$ in $\mathfrak{S}_{\lambda}^*$, then 
	\begin{equation}
		\label{eq:euler-lagrange-phi} 
		k \widetilde{\phi}''  = \widetilde{\phi} - \widetilde{\theta},
		\quad \text{a.e. on }[0,1].
	\end{equation}
\end{prop}

\noindent Some preliminary definitions and results are required 
to prove Proposition \ref{prop:local_EulerLagrange}.
In fact,
an important step of the proof is showing that 
$\widetilde\phi$ is sufficiently smooth on the intervals
$[0,x_\lambda]$ and $[x_\lambda,1]$
(see Lemma \ref{lem:regularity_Ploc_minimum});
this can be achieved by 
exploiting the techniques given by \cite{MarinoScolozzi83},
which have been used in different contexts
to achieve the desired regularity of constrained minimizers
(see e.g. \cite{Corona2020JFPT,ggp}).
An important consequence of Proposition \ref{prop:local_EulerLagrange}
and of the special definition of $\phi^*_\lambda$
is that the constrained minimizer 
$(\widetilde\phi,\widetilde\theta) \in \mathfrak{S}^*_\lambda$
is such that 
$\widetilde\phi$ equals $\phi^*_\lambda$
on $0$ and, at most, on $x_\lambda$:
in other words
$\widetilde\phi(x) < \phi^*_\lambda(x)$ for all $x \ne 0,x_\lambda$.
%After having established the regularity of constrained minimizers,
%the special form of the constraint $\phi_{\lambda}^*$
%allows us to prove Proposition \ref{prop:local_EulerLagrange}. 

\begin{definition}
	For every $\phi \in \mathfrak{C}_{\lambda}^*$,
	we define 
	the set of infinitesimal admissible variations of $\phi$ in
	$\mathfrak{C}_{\lambda}^*$, denoted by $\mathcal{V}_{\lambda}^*$,
	the set
	\begin{equation}
		\label{eq:calV-phi}
		\mathcal{V}_{\lambda}^*(\phi)\coloneqq 
		\bigg\{
			\xi \in H^{1}_*([0,1]):
			\xi(x) \le 0 \text{ if } \phi(x) = \phi_{\lambda}^*(x)
		\bigg\}.
	\end{equation}
\end{definition}

\begin{lemma}
	\label{lem:algebraic_theta}
	Let $(\widetilde{\phi},\widetilde{\theta})$ be a global minimizer
	of $F_{b,k}$ in $\mathfrak{S}_{\lambda}^*$, then 
	\begin{equation}
		\label{eq:algebraic_theta} 
		\widetilde{\phi} - \widetilde\theta = - b (1-x)\cos\widetilde\theta,
		\quad \text{a.e. on }[0,1].
	\end{equation}
\end{lemma}
\begin{proof}
	For every $\eta\in C^\infty_0([0,1])\subset L^2([0,1])$ we have
	\[
		\text{d}F_{b,k}(\widetilde\phi,\widetilde\theta)[0,\eta]=
		-\int_{0}^{1}
		\bigg(\widetilde{\phi}-\widetilde{\theta} + b(1-x)\cos\widetilde\theta\bigg)\eta~ \text{d}x
		=0,
	\]
	from which \eqref{eq:algebraic_theta} follows.
\end{proof}
\begin{lemma}
	\label{lem:tildephi_ge_32pi}
	Let $(\widetilde{\phi},\widetilde{\theta})$ be a global minimizer
	of $F_{b,k}$ on $\mathfrak{S}_{\lambda}^*$, then 
	\begin{equation*}
		\widetilde\phi(x) \ge -\frac{3}{2}\pi,\quad \forall x \in [0,1].
	\end{equation*}
\end{lemma}
\begin{proof}
	Reasoning by contradiction, if there exists $x_1 \in ]0,1]$ such that 
	$\widetilde\phi(x_1) < -\frac{3}{2}\pi$,
	by continuity there exists $x_0 \in ]0,x_1[$ such that
	$\widetilde\phi(x_0) = -\frac{3}{2}\pi$.
	Therefore, we can define
	the functions $\phi_1,\theta_1\colon [0,1] \to \mathbb{R}$ as
	\[
		\phi_1(x) = 
		\begin{cases}
			\widetilde\phi(x), &		\mbox{if } x \le x_0, \\
			-\frac{3}{2}\pi, &	\mbox{if }  x> x_0,
		\end{cases}
		\quad\text{and}\quad
		\theta_1(x) = 
		\begin{cases}
			\widetilde\theta(x), &		\mbox{if } x \le x_0, \\
			-\frac{3}{2}\pi, &	\mbox{if }  x> x_0.
		\end{cases}
	\]
	Since $\widetilde\phi(x_1) < -\frac{3}{2}\pi$, we have
	$
	\int_{x_0}^{1}\widetilde\phi'^2\text{d}x
	\ge
	\int_{x_0}^{x_1}\widetilde\phi'^2\text{d}x
	> 0,
	$
	hence
	\begin{multline*}
		F_{b,k}(\widetilde\phi,\widetilde\theta) - F_{b,k}(\phi_1,\theta_1) 
		=
		\int_{x_0}^{1} \bigg(
			k\frac{\widetilde\phi'^2}{2} + \frac{(\widetilde\phi-\widetilde\theta)^2}{2} 
		\bigg)\text{d}x\\
		+
		\int_{x_0}^{1} b(1-x)(1 - \sin\widetilde\theta)\text{d}x
		\ge
		\int_{x_0}^{1}
		k\frac{\widetilde\phi'^2}{2}
		\text{d}x >0,
	\end{multline*}
	contradicting the minimality of $(\widetilde\phi,\widetilde\theta)$.
\end{proof}
\begin{lemma}
	\label{lem:noTouch-after-xlambda}
	If $(\widetilde{\phi},\widetilde{\theta})$ is a global minimizer of
	$F_{b,k}$ in $\mathfrak{S}_{\lambda}^*$, then
	\[
		\widetilde\phi(x) <  -\pi,
		\quad \forall x \in (x_\lambda,1].
	\]
\end{lemma}
\begin{proof}
	Being a global minimizer of $F_{b,k}$,
	the restriction of 
	$(\widetilde{\phi},\widetilde{\theta})$
	on the interval $[x_\lambda,1]$,
	is a global minimizer for the functional
	\[
		(\phi,\theta)\mapsto
		\int_{x_\lambda}^1 
		\bigg(
			k\frac{\phi'^2}{2} +\frac{(\phi-\theta)^2}{2} - b(1-x)\sin\theta
		\bigg)\text{d}x
	\]
	with the conditions $\phi(x_\lambda) = \widetilde\phi(x_\lambda) \le -\pi$
	and $\phi(x)\le - \pi$.
	As a consequence, the pair
	$(\widetilde\phi_1,\widetilde\theta_1) \coloneqq (\pi - \widetilde\phi,\pi - \widetilde\theta)$
	is the global minimizer of the same functional
	under the conditions $\phi(x_\lambda) = \pi -\widetilde\phi(x_\lambda) \ge 0$
	and $\phi(x) \ge 0$.
	Thus, by Lemma \ref{lem:monotonicity},
	$\widetilde\phi_1$ is strictly increasing,
	so $\widetilde\phi$ is strictly decreasing.
	Since $\widetilde\phi(x)\le \phi^*_\lambda(x) = -\pi$
	for all $x\in [x_\lambda,1]$,
	we obtain the thesis.
\end{proof}

\begin{lemma}
	\label{lem:regularity_Ploc_minimum}
	If $(\widetilde{\phi},\widetilde{\theta})$ is a global minimizer of
	$F_{b,k}$ in $\mathfrak{S}_{\lambda}^*$, then
	\begin{equation*}
		\widetilde\phi|_{[0,x_\lambda]} \in W^{2,\infty}([0,x_\lambda],\mathbb{R})
	\end{equation*}
	and
	\begin{equation}
		\label{eq:regularity_xlambda1}
		\widetilde\phi|_{[x_\lambda,1]} \in W^{2,\infty}([x_\lambda,1],\mathbb{R}).
	\end{equation}
\end{lemma}
\begin{proof}
	Thanks to Lemma \ref{lem:noTouch-after-xlambda},
	every $\xi \in C^\infty([0,1],\mathbb{R})$
	with compact support in $(x_\lambda,1)$ is an
	admissible variation.
	As a consequence,
	the regularity indicated by \eqref{eq:regularity_xlambda1}
	can be obtained by standard arguments.

	\noindent Therefore, from now on in this proof,
	we restrict our study on the interval $[0,x_\lambda]$.
	For the sake of presentation,
	we simply write $\widetilde\phi$ instead of 
	$\widetilde\phi|_{[0,x_\lambda]}$
	and, similarly, the sets $\mathfrak{C}^*_\lambda$ and
	$\mathcal{V}^*_\lambda(\widetilde\phi)$
	have to be meant as defined on the interval $[0,x_\lambda]$. 
	Since $(\widetilde\phi,\widetilde\theta)$ is a global minimizer
	for $F_{b,k}$,
	$\widetilde\phi$ is a global minimizer for the functional
	$G\colon \mathfrak{C}_{\lambda} \to \mathbb{R}$ defined as 
	\begin{equation*}
		G(\phi) \coloneqq
		\int_{0}^{x_\lambda}
		\bigg(
			\frac{k}{2}(\phi')^2 + \frac{(\phi - \widetilde{\theta})^2}{2} 
		\bigg)\text{d}x,
	\end{equation*}
	so it satisfies
	\begin{equation*}
		\text{d}G(\widetilde{\phi})[\xi] = 
		\int_0^{x_\lambda}
		\bigg(
			k\phi'\xi' + (\phi - \theta)\xi
		\bigg)dx \ge 0,
		\qquad
		\forall \xi \in \mathcal{V}_{\lambda}^*(\widetilde{\phi}).
	\end{equation*}

	\noindent Set $y = \widetilde{\phi} - \phi^*$.
	Since $\phi^{*}$ is of class $C^2$ on $[0,x_\lambda]$, our thesis can be obtained
	by proving that $y \in W^{2,\infty}([0,x_\lambda])$,
	thus by showing that $y' \in W^{1,\infty}([0,x_\lambda])$.
	Defining the function $z\colon [0,x_\lambda]\to \mathbb{R}$ as 
	\[
		z \coloneqq y + \phi^* - \widetilde{\theta} - (\phi^*)'',
	\]
	we can write the differential of $G$ as follows: 
	\begin{equation*}
		\begin{multlined}
			\text{d}G(\widetilde{\phi})[\xi] = 
			\int_{0}^{x_\lambda}
			\bigg[
				k(y' + (\phi^*)')\xi' + (y + \phi^* - \widetilde{\theta})\xi
			\bigg] \text{d}x \\
			= 
			\int_{0}^{x_\lambda}
			\bigg[
				ky'\xi' + (y + \phi^* - \widetilde{\theta} - (\phi^*)'')\xi
			\bigg] \text{d}x
			= 
			\int_{0}^{x_\lambda}
			\bigg(
				k y' \xi' + z \xi
			\bigg) \text{d}x.
		\end{multlined}
	\end{equation*}
	For all $x \in [0,x_\lambda]$,
	$\widetilde{\phi}(x) \le \phi_{\lambda}^{*}(x)$
	and, by Lemma \ref{lem:tildephi_ge_32pi}, $\widetilde\phi(x)\ge -3/2 \pi$.
	Hence, $\widetilde\phi$ is bounded and, by \eqref{eq:algebraic_theta}, 
	$\widetilde{\theta} \in L^\infty([0,x_\lambda])$.
	As a consequence, $z \in L^{\infty}([0,x_\lambda])$.
	Let us define
	\[
		J = \bigg\{
			x \in [0,x_\lambda]: \widetilde{\phi}(x) = \phi^*(x)
		\bigg\} \cup \{0,x_\lambda\}
		\quad \text{and} \quad 
		I = [0,x_\lambda] \setminus J.
	\]
	The set $I$ is an open set,
	hence it is a countable union of pairwise disjoint open intervals and we can write
	\[
		I = \bigcup_{i \in A}]a_i,b_i[,
	\]
	where $A$ is a countable set.
	Let us consider an arbitrary scalar field $\nu \in W^{1,2}_0([0,x_\lambda])$
	such that $\nu(x) = 0$ for all $x \in J$.
	As a consequence,
	both $\nu$ and $-\nu$ are
	infinitesimal admissible variations of $\widetilde{\phi}$ in $\mathfrak{C}_{\lambda}$
	and we have
	\[
		\text{d}G(\widetilde{\phi})[\nu] 
		= 
		\sum_{i \in A}
		\int_{a_i}^{b_i} \bigg(ky' \nu' + z \nu \bigg) \text{d}x
		= 0,
	\]
	hence, for the arbitrariness of $V$,
	\[
		\int_{a_i}^{b_i} \bigg(ky' \nu' + z \nu \bigg) \text{d}x  = 0, \qquad \forall i \in A.
	\]
	By a standard argument,
	we obtain that $y'$ is absolutely continuous in $I$ and it satisfies
	\begin{equation}
		\label{eqn:regularityProof1}
		-k y'' + z = 0,
		\qquad \text{a.e. in } I.		
	\end{equation}
	%Now we are going to analyze the behaviour on $J$.
	For an arbitrary $\xi \in W^{1,2}_0([0,x_\lambda])$,
	if we set $\eta(x) = \max\{\xi(x), 0 \}$, then
	\[
		\zeta(x) = \xi(x) - \eta(x) \in \mathcal{V}_{\lambda}^*(\widetilde\phi),
	\]
	hence
	\[
		\text{d} G(\widetilde\phi)[\zeta]\ge 0.
	\]
	By \eqref{eqn:regularityProof1}, partial integration reduces to
	\[
		\int_I (ky'\eta' + z \eta)\text{d}x = 
		\sum_{i \in A} (y'(b_i)\eta(b_i)- y'(a_i)\eta(a_i)).
	\]
	Since $y = 0$ in $J$ and $y < 0$ in $I$ we have
	\[
		y'(b_i) \ge 0, \qquad y'(a_i) \le 0, \qquad \forall i \in A,
	\]
	except for $y'(0)$ and $y'(1)$, but in that cases $\eta(0) = \eta(x_\lambda)= 0$.
	As a consequence, 
	\[
		\int_I (ky'\eta' + z \eta)\text{d}x \ge 0,
	\]
	and we have
	\[
		\begin{multlined}
			0 \le \text{d}G(\widetilde{\phi})[\zeta] = 
			\int_0^{x_\lambda} (ky'\xi' + z \xi)\text{d}x
			- \int_J (ky'\eta' + z \eta)\text{d}x
			- \int_I (ky'\eta' + z \eta)\text{d}x
			\\
			\le 
			\int_0^{x_\lambda} (ky'\xi' + z \xi)\text{d}x
			- \int_J (ky'\eta' + z \eta)\text{d}x,
		\end{multlined}	
	\]
	hence
	\[
		\int_0^{x_\lambda} (ky'\xi' + z \xi)\text{d}x \ge 
		\int_J (ky'\eta' + z \eta)\text{d}x.
	\]
	Since $y = 0$ on $J$, then $y'= 0$ a.e. on $J$ (cf. \cite{GT} Lemma 7.7)
	and we obtain
	\begin{equation}
		\label{eqn:regularityProof2}
		\int_0^{x_\lambda} (ky'\xi' + z \xi)\text{d}x \ge \int_J z \eta\ \text{d}x.
	\end{equation}
	By \eqref{eqn:regularityProof2}, recalling that $|\eta| \le |\xi|$,  we obtain 
	\[
		\bigg| \int_0^{x_\lambda} (ky'\xi' + z \xi)\text{d}x \bigg|
		\le \lVert z \rVert_{L^\infty}\lVert\xi\rVert_{L^\infty},
	\]
	whence
	\[
		\left|\int_0^{x_\lambda} ky'\xi'~ \text{d}x \right|
		\le \left| \int_0^{x_\lambda} (ky'\xi' + z \xi)\text{d}x \right|
		+ \left|\int_0^{x_\lambda}z\xi~ \text{d}x \right|
		\le 2\lVert z\rVert_{L^\infty}\lVert\xi\rVert_{L^\infty}.
	\]
	Since $\xi(0) = 0$, then there exists a constant $c_1$ 
	such that $\lVert\xi\rVert_{L^{\infty}} \le c_1 \lVert\xi'\rVert_{L^1}$ and we obtain
	\[
		\left|\int_0^{x_\lambda} ky'\xi'~ \text{d}x \right| 
		\le 2c_1 \lVert z\rVert_{L^\infty}\lVert\xi'\rVert_{L^1},
		\qquad \forall \xi' \in L^{1}([0,x_\lambda]).
	\]
	Hence, $y' \in L^{\infty}([0,x_\lambda])$ by the Riesz representation theorem.
	Using again \eqref{eqn:regularityProof2},
	there exists a constant $c_2$ such that 
	\[
		\left|\int_0^{x_\lambda} ky'\xi'~ \text{d}x \right|
		\le c_2 \lVert z\rVert_{L^\infty}\lVert\xi\rVert_{L^1}.
	\]
	By a standard argument (see, for instance, \cite[Proposition 8.3]{Brezis}),
	this suffices to conclude that
	$y' \in W^{1,\infty}([0,x_\lambda])$.
\end{proof}

\noindent Now we are ready to prove Proposition \ref{prop:local_EulerLagrange}.
\begin{proof}[\textbf{Proof of Proposition \ref{prop:local_EulerLagrange}}]
	Let $\mathcal{V}_{\lambda}^*(\widetilde\phi)$ be the set of all
	admissible infinitesimal variations of $\widetilde\phi$ in $\mathfrak{C}_{\lambda}$,
	defined as in \eqref{eq:calV-phi}.
	Since $(\widetilde{\phi},\widetilde{\theta})$ is a global minimizer, 
	\begin{equation}
		\label{eq:local-EL-proof1}
		\text{d}F_{b,k}(\widetilde{\phi},\widetilde{\theta})[\xi,0]
		= \int_{0}^{1} \bigg(
			k\widetilde{\phi}'\xi' + (\widetilde{\phi} - \widetilde{\theta})\xi
		\bigg) dx
		\ge 0,
		\quad \forall \xi \in \mathcal{V}_{\lambda}^*(\widetilde{\phi}).
	\end{equation}
	By Lemma \ref{lem:noTouch-after-xlambda},
	every function of class $C^\infty$
	with compact support in $(x_\lambda,1)$
	belongs to $\mathcal{V}_\lambda^*(\widetilde\phi)$.
	As a consequence,
	by a standard argument we obtain that 
	\[
		k\widetilde\phi'' = \widetilde\phi - \widetilde\theta,
		\quad \text{a.e. on }[x_\lambda,1],
	\]
	and we can reduce our analysis on the interval $[0,x_\lambda]$.
	Let us now consider a variation in $\mathcal{V}_{\lambda}^*(\widetilde{\phi})$
	with compact support in $(0,x_\lambda)$.
	By Lemma \ref{lem:regularity_Ploc_minimum},
	$\widetilde{\phi} \in W^{2,\infty}([0,x_\lambda])$ 
	so we can integrate by parts \eqref{eq:local-EL-proof1}
	and obtain
	\begin{equation}
		\label{eqn:GlobMinFreeProof2}
		- k \widetilde{\phi}'' +(\widetilde{\phi} - \widetilde{\theta}) 
		= \tau(x) \le 0,
		\quad \text{a.e. on }[0,x_\lambda],
	\end{equation}
	where $\tau(x) = 0$ if $\widetilde{\phi}(x) < \phi^*_\lambda(x)$.
	Set
	\[
		J = \left\{
			x \in [0,x_\lambda]: \widetilde{\phi}(x) = \phi^*_\lambda(x)
		\right\}.
	\]
	Since $\widetilde{\phi} \in W^{2,\infty}([0,x_\tau])$, 
	using \cite[Lemma 7.7]{GT} we obtain that 
	$
	\widetilde{\phi}'' = (\phi^*_\lambda)''
	$
	a.e. on $J$.
	Recalling also \eqref{eq:algebraic_theta},
	we obtain
	\begin{multline*}
		% \label{eqn:GlobMinFreeProof3}
		\tau = 	-k (\phi^*_\lambda)'' - b(1-x)\cos\widetilde\theta 
		\ge -k(\phi^*_\lambda)'' - b(1-x)\\
		= b(1-x) +  k - b(1-x)
		=  k > 0,
		\quad \text{a.e. on }J.
	\end{multline*}
	As a consequence,
	from \eqref{eqn:GlobMinFreeProof2}
	we deduce that $J$ is a set of measure zero and \eqref{eq:euler-lagrange-phi} follows.
\end{proof}

\noindent Using Proposition \ref{prop:local_EulerLagrange}
and exploiting again the properties of $\phi_{\lambda}^*$,
we obtain that $\widetilde\phi$ can coincide with $\phi_{\lambda}^*$ 
only at $0$ and $x_\lambda$.
More formally, we have the following result.

\begin{cor}
	\label{cor:notTouch}
	Let $(\widetilde{\phi},\widetilde{\theta})$ be a global minimizer
	of $F_{b,k}$ in $\mathfrak{S}_{\lambda}^*$.
	Then 
	\[
		\widetilde\phi(x) < \phi_{\lambda}^*(x),
		\quad \forall x \ne 0,x_{\lambda}.
	\]
\end{cor}

\begin{proof}
	Using again Lemma \ref{lem:noTouch-after-xlambda},
	it suffices to prove that 
	$\widetilde\phi(x)<\phi^*_\lambda(x)$
	for all $x \in (0,x_\lambda)$.
	Seeking a contradiction,
	let $\bar{x} \in (0,x_\lambda)$
	be such that $\widetilde\phi(\bar{x}) = \phi_{\lambda}^*(\bar{x})$.
	Since $\widetilde\phi(x)\le \phi_{\lambda}^*(x)$ for all $x \in [0,x_\lambda]$
	and, by Lemma \ref{lem:regularity_Ploc_minimum},
	$\widetilde\phi\in W^{2,\infty}([0,x_\lambda])\subset C^1([0,x_\lambda])$
	we have $\widetilde\phi'(\bar{x}) = {\phi_{\lambda}^*}'(\bar{x})$.
	Hence, we obtain
	\begin{equation}
		\label{eq:th42-proof1}
		0 \ge \widetilde\phi(x) - \phi_{\lambda}^*(x)
		= \int_{\bar{x}}^{x} \bigg( 
			\int_{\bar{x}}^{s} (\widetilde{\phi} - \phi_{\lambda}^*)''(\tau) \text{d}\tau
		\bigg) \text{d}s,
		\quad \forall x \in [\bar{x}, x_\lambda].
	\end{equation}
	Since $\phi_{\lambda}^*$ is defined by \eqref{eq:def_phiStar},
	we have
	\[
		{\phi_{\lambda}^*}''(x)  =
		- \left( \lambda(1-x) +1 \right),
		\qquad \text{on }[0,x_\lambda],
	\]
	while \eqref{eq:euler-lagrange-phi} and \eqref{eq:algebraic_theta}
	imply
	\[
		\widetilde{\phi}''(x) = -\frac{b}{k}(1-x)\cos\widetilde\theta
		\ge -\lambda(1-x)
		\quad \text{a.e. on }[0,x_\lambda].
	\]
	As a consequence, 
	from \eqref{eq:th42-proof1} we obtain
	that for every $x \in ]\bar{x},x_\lambda]$
	we have
	\begin{equation*}
		0 
		\ge \int_{\bar{x}}^{x} \bigg( \int_{\bar{x}}^{s}
			\bigg(
				- \lambda(1-\tau)+ \lambda(1-\tau)+1 
			\bigg)
		\text{d}\tau \bigg) \text{d}s = 
		\frac{(x - \bar{x})^2}{2} > 0,
	\end{equation*}
	which is absurd.

\end{proof}
\begin{remark}
	\label{rem:onlySecondDeriv}
	We notice that the proofs of Proposition \ref{prop:local_EulerLagrange}
	and of Corollary \ref{cor:notTouch} rely only on the
	second order derivative of $\phi^*_\lambda$.
	As a consequence, if we substitute this constraint
	with another function with the same second order derivative
	we obtain analogous results.
	This observation will be useful in the final part of our work,
	when we will use a $\Gamma$--convergence argument to show the
	local minimality of $(\tilde\phi,\tilde\theta)$.
\end{remark}

\subsection{Convergence to minimizers of the Euler beam}
\label{sub:32}

The following results are needed to show that,
also for $x=x_{\lambda}$,
the minimizer under $\phi^*_{\lambda}$ does not ``touch''
the constraint.
This requires considerably more effort,
and it is achieved through a comparison with the easier cases represented by the functionals describing the nonlinear Euler beam under uniformly distributed and concentrated load.

\noindent Recalling the definition of $\mathfrak{C}^*_{\lambda}$
given in \eqref{eq:def-C*lambda},
we define the functional $F_\lambda\colon \mathfrak{C}^*_{\lambda}\to\mathbb{R}$
as
\begin{equation}
	\label{eq:def-Flambda}
	F_{\lambda}(\phi) \coloneqq \int_{0}^{1}
	\left(
		\frac{|\phi'|^2}{2} - \lambda(1-x)\sin\phi
	\right)\text{d}x,
\end{equation}
which corresponds to the energy functional
of a nonlinear Euler beam under distributed load
(see e.g. \cite{Euler__}).
Denoting by $\widetilde\phi_\lambda$ its minimizer, 
the main results of this subsection are the following:
\begin{itemize}
	\item if $b$ is sufficiently small and $b/k = \lambda$,
		then the global minimizer $(\widetilde\phi,\widetilde\theta)$ 
		of $F_{b,k}$ in $\mathfrak{S}^*_\lambda$
		is such that $\lVert\widetilde\phi(x) - \widetilde\phi_\lambda(x)\rVert$
		is arbitrarily small:
		in other words, the solutions of the problem of a nonlinear Timoshenko beam
		are similar to the ones of a nonlinear Euler beam;
	\item if $\lambda$ is sufficiently large, then $\widetilde\phi_\lambda(x_\lambda)$
		is strictly less then $-\pi$:
		this result will be achieved by a limit process
		that can get rid of the autonomous component in the functional
		\eqref{eq:def-Flambda}.
\end{itemize}

\begin{remark}
	\label{remF_0}
	The arguments used in the proofs
	of Proposition \ref{prop:local_EulerLagrange},
	of Lemma \ref{lem:regularity_Ploc_minimum},
	and of Corollary \ref{cor:notTouch}
	can be easily applied to $F_{\lambda}$.
	Therefore, if $\widetilde{\phi}_{\lambda}$ is a minimizer of $F$ in
	$\mathfrak{C}_{\lambda}^*$, we have 
	\begin{equation*}
		\widetilde\phi_{\lambda}''+\lambda(1-x)\cos\widetilde\phi = 0,
		\quad \text{a.e. on } [0,1],
	\end{equation*}
	and
	\begin{equation*}
		\widetilde\phi_{\lambda}(x) < \phi_{\lambda}^*(x)
		\quad \forall x \ne 0,x_{\lambda}.
	\end{equation*}
\end{remark}
\begin{prop}
	\label{prop:convergence}
	Fix $\lambda_0 \in \mathbb{R}$ and
	let $(b_n,k_n)_{n \in \mathbb{N}} \subset \mathbb{R}^+ \times \mathbb{R}^+$
	be such that 
	\[
		\lim_{n \to \infty} b_n = 0
		\quad\text{and}\quad
		\frac{b_n}{k_n} = \lambda_0, \quad \forall n \in \mathbb{N}.
	\]
	Let $(\widetilde\phi_n,\widetilde\theta_n) \in \mathfrak{S}_{\lambda_0}^*$
	be the sequence of corresponding minimizers of $F_{b_n,k_n}$.
	Then, up to considering a subsequence,
	$\widetilde\phi_n$ and $\widetilde\theta_n$
	converge in the $C^1$ norm and a.e., respectively,
	to a function $\widetilde\phi_{\lambda_0}$
	which is a global minimizer of
	the functional
	$F_{\lambda_0}\colon \mathfrak{C}_{\lambda_0}^*\to \mathbb{R}$.
\end{prop}

\begin{proof}
	By Proposition \ref{prop:local_EulerLagrange} and Lemma \ref{lem:algebraic_theta},
	for every $n\in \mathbb{N}$ the pair $(\widetilde\phi_n,\widetilde\theta_n)$
	satisfies almost everywhere the following system of equations:
	\[
		\begin{cases}
			-k_n\widetilde\phi_n'' + \widetilde\phi_n - \widetilde\theta_n = 0,\\
			\widetilde\phi_n - \widetilde\theta_n = -b_n(1-x)\cos\widetilde\theta_n.
		\end{cases}
	\]
	As a consequence,
	\[
		\widetilde\phi_n'' = -\frac{b_n}{k_n}(1-x)\cos\widetilde\theta_n
		\le \lambda_0,
	\]
	so the sequence $\widetilde\phi_n''$ is equibounded 
	with respect to the norm of $L^\infty([0,1])$.
	By the Ascoli-Arzelà theorem,
	$\widetilde\phi_n$ converges,
	up to subsequences, in the $C^1([0,1])$ norm to a function 
	$\widetilde\phi_{\lambda_0} \in \mathfrak{C}_{\lambda_0}$.
	By hypothesis, $b_n \to 0$,
	and using
	\[
		\widetilde\theta_n = \widetilde\phi_n - b_n(1-x)\cos\widetilde\theta_n,
		\quad \text{a.e. in $[0,1]$},
	\]
	we obtain that $\widetilde\theta$ converges to $\widetilde\phi_{\lambda_0}$ a.e..
	Therefore,
	by the dominated convergence theorem we have
	\begin{multline}
		\label{eq:convergence_proof1}
		\lim_{n\to\infty}
		\frac{1}{k_n}
		F_{b_n,k_n}(\widetilde\phi_n,\widetilde\theta_n)\\
		= \lim_{n\to\infty}
		\int_0^1
		\left(
			\frac{|\widetilde\phi_n'|^2}{2}
			+ \frac{b_n^2 (1-x)^2\cos^2\widetilde\theta_n}{2 k_n}
			- \frac{b_n}{k_n}(1-x)\sin\widetilde\theta_n
		\right)\text{d}x\\
		=
		\int_0^1
		\left(
			\frac{|\widetilde\phi_{\lambda_0}'|^2}{2}
			- \lambda_0(1-x)\sin\widetilde\phi_{\lambda_0}
		\right)\text{d}x
		=F_{\lambda_0}(\widetilde\phi_{\lambda_0}).
	\end{multline}
	It remains to show that $\widetilde\phi_{\lambda_0}$ is a minimizer for 
	$F_{\lambda_0}$ in $\mathfrak{C}_{\lambda_0}^*$.
	By contradiction, let $\psi \in C^*_{\lambda_0}$ be
	such that $F_{\lambda_0}(\psi)< F_{\lambda_0}(\widetilde\phi_{\lambda_0})$.
	Since
	\[
		\lim_{n\to\infty}\frac{1}{k_n}F_{b_n,k_n}(\psi,\psi) = F_{\lambda_0}(\psi)
	\]
	and \eqref{eq:convergence_proof1} holds,
	there exist  $\epsilon > 0$
	and $n$ such that
	\[
		\frac{1}{k_n}F_{b_n,k_n}(\widetilde\phi_n,\widetilde\theta_n)
		> F_{\lambda_0}(\widetilde\phi_{\lambda_0}) - \epsilon
		> F_{\lambda_0}(\psi) + \epsilon
		> \frac{1}{k_n}F_{b_n,k_n}(\psi,\psi),
	\]
	contradicting the minimality of $(\widetilde\phi_n,\widetilde\theta_n)$.
\end{proof}

\noindent Proposition \ref{prop:convergence} entails that
our aim is to study the behaviour of 
the minimizer $\widetilde\phi_\lambda \in \mathfrak{C}_\lambda^*$ of $F_{\lambda}$ 
as $\lambda$ goes to infinity.
In particular, by Remark \ref{remF_0},
we need to prove that if $\lambda$ is sufficiently large then
$\widetilde\phi_\lambda(x_\lambda) < -\pi$.
The next result provides a necessary condition 
for a function whose graph passes through $(x_\lambda,-\pi)$
to be a minimizer. This condition involves the left and right derivatives at $x_\lambda$. 
\begin{lemma}
	\label{lem:phiTouch}
	Let $\widetilde\phi_\lambda \in \mathfrak{C}^*_\lambda$
	be a minimizer of $F_{\lambda}$.
	If $\widetilde\phi_\lambda(x_\lambda) = -\pi$,
	then
	\begin{equation}
		\label{eq:opt-condition}
		\widetilde\phi_\lambda'(x_\lambda^-) 
		\le \widetilde\phi_\lambda'(x_\lambda^+).
	\end{equation}
\end{lemma}
\begin{proof}
	If $\widetilde\phi_\lambda$ is a minimizer, then 
	for every $\xi \in C^\infty_0([0,1])$ such that
	\[
		\xi(x) \le 0, \quad \forall x \in [0,1]
		\quad\text{and}\quad
		\xi(x_\lambda) < 0,
	\]
	we have $\text{d}F_\lambda(\widetilde\phi_\lambda)[\xi] \ge 0$, hence
	\[
		\int_0^1 \left(
			\widetilde\phi_\lambda'\xi' - \lambda(1-x)\cos\widetilde\phi_\lambda\ \xi
		\right)\text{d}x \ge 0.
	\]
	By Remark \ref{remF_0},
	$\widetilde\phi_\lambda$ does not coincide with 
	$\phi^*_\lambda$ except in $0$ and,
	by hypothesis, in $x_\lambda$.
	Therefore, it satisfies the Euler-Lagrange equation
	both in $(0,x_\lambda)$ and in $(x_\lambda,1)$
	and an integration by parts leads to
	\[
		\text{d}F_\lambda(\widetilde\phi_\lambda)[\xi]
		= \left(\widetilde\phi_\lambda'(x_\lambda^-) 
			-\widetilde\phi_\lambda'(x_\lambda^+) 
		\right)\xi(x_\lambda) \ge 0.
	\]
	By the arbitrariness of $\xi(x_\lambda) < 0$,
	we obtain \eqref{eq:opt-condition}.
\end{proof}

\noindent Due to Lemma \ref{lem:phiTouch},
it becomes important to estimate the behaviour of the 
left and right derivatives at $x_\lambda$
of the minimizer of $F_\lambda$ 
among all the functions in $\mathfrak{C}_\lambda$
whose graph passes through $(x_\lambda,-\pi)$.
To this aim, we separately study 
the functional in the two intervals $[0,x_\lambda]$
and $[x_\lambda,1]$
and we define the following sets
\[
	\mathfrak{L}_\lambda \coloneqq
	\bigg\{
		\phi \in H^1([0,x_\lambda], \mathbb{R}): \phi(x) \le \phi^*_\lambda(x)
		\ \forall x \in [0,x_\lambda],
		\phi(0) = 0 \text{ and }
		\phi(x_\lambda) = - \pi
	\bigg\}
\]
and 
\[
	\mathfrak{R}_\lambda \coloneqq
	\bigg\{
		\phi \in H^1([x_\lambda,1],\mathbb{R}): \phi(x) \le - \pi
		\ \forall x \in [x_\lambda,1],
		\phi(x_\lambda) = - \pi
	\bigg\}.
\]
On them, we define the functionals
$ L_\lambda\colon \mathfrak{L}_{\lambda} \to \mathbb{R} $
and
$ R_\lambda\colon \mathfrak{R}_{\lambda} \to \mathbb{R} $
as
\begin{equation}
	\label{eq:def-L}
	L_\lambda(\phi) \coloneqq
	\int_{0}^{x_\lambda}
	\left(\frac{\phi'^2}{2} - \lambda(1-x)\sin\phi\right)\text{d}x
\end{equation}
and
\begin{equation}
	\label{eq:def-R}
	R_\lambda(\phi) \coloneqq
	\int_{x_\lambda}^1
	\left(\frac{\phi'^2}{2} - \lambda(1-x)\sin\phi\right)\text{d}x.
\end{equation}

\noindent Let $\ell_\lambda \in \mathfrak{L}_\lambda$
and $r_\lambda \in \mathfrak{R}_{\lambda}$
be the minimizers of $L_\lambda$ and $R_\lambda$,
respectively.
Then a function $\psi_\lambda \in \mathfrak{C}^*_\lambda$
such that $\psi_\lambda(x_\lambda) = - \pi$
is a minimizer of $F_\lambda$ if and only if 
\begin{equation}
	\label{eq:def-psi}
	\psi_\lambda(x) = 
	\begin{cases}
		\ell_\lambda(x), &	\mbox{if } x \in [0,x_\lambda], \\
		r_\lambda(x), &	\mbox{if } x \in [x_\lambda,1],
	\end{cases}
\end{equation}
and, by Lemma \ref{lem:phiTouch}, if 
\begin{equation}
	\label{eq:opt-cond-lr}
	\ell_\lambda'(x_\lambda) \le r_\lambda'(x_\lambda).
\end{equation}
As a consequence, proving that \eqref{eq:opt-cond-lr} 
does not hold for $\lambda$ sufficiently large
implies that the minimizers of $F_\lambda$ in $\mathfrak{C}^*_\lambda$
does not pass through $(x_\lambda,-\pi)$.
%Then, using also Proposition \ref{prop:convergence},
%we will be able to prove Theorem \ref{teo:local-min}.
Fig.~\ref{fig:Euler15} and Fig.~\ref{fig:Euler100} show 
$\ell_\lambda$ and $r_\lambda$
for $\lambda = 15$ and $\lambda = 100$, respectively.
In Fig.~\ref{fig:Euler15}, it can be noticed that \eqref{eq:opt-cond-lr}
holds, and indeed $\psi_\lambda$ defined as in \eqref{eq:def-psi}
is a global minimizer of $F_\lambda$ in $\mathfrak{C}^*_\lambda$.
In Fig.~\ref{fig:Euler100}, it can be noticed that \eqref{eq:opt-cond-lr}
does not holds, so the graph of the global minimizer $\widetilde\phi_\lambda$
does not pass through the point $(x_\lambda,-\pi)$. 

\begin{figure}[h]
	\centering
	\includegraphics[width=0.6\linewidth]{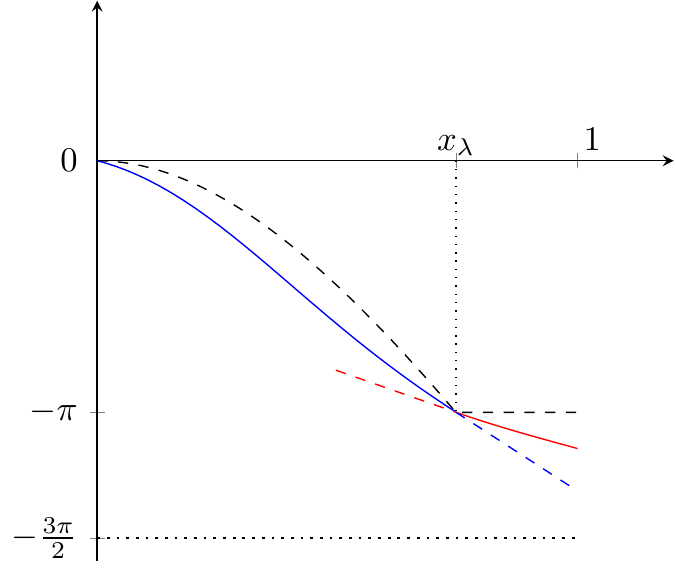}
	\caption{
		The functions $\ell_\lambda$ and $r_\lambda$ for  $\lambda = 15$;
		(numerically evaluated);
		it can be noticed that $\ell_\lambda'(x_\lambda) < r_\lambda'(x_\lambda)$ and the graph of
		$\widetilde\phi_\lambda$ passes through $(x_\lambda,-\pi)$.
	}
	\label{fig:Euler15}
\end{figure}
\begin{figure}[h]
	\centering
	\includegraphics[width=0.6\linewidth]{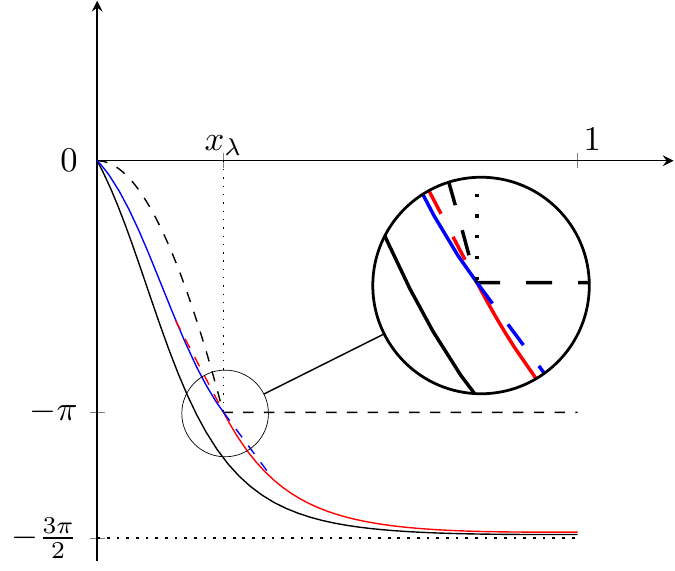}
	\caption{
		The functions $\ell_\lambda$ and $r_\lambda$ for  $\lambda = 100$
		(numerically evaluated);
		it can be noticed that $\ell_\lambda'(x_\lambda) > r_\lambda'(x_\lambda)$,
		so by Lemma \ref{lem:phiTouch}
		a function whose graph passes through $(x_\lambda,-\pi)$
		cannot be a minimizer of $F_\lambda$ in $\mathfrak{C}^*_\lambda$.
	}
	\label{fig:Euler100}
\end{figure}

\noindent Before giving the estimates for 
$\ell_\lambda'(x_\lambda)$
and 
$r_\lambda'(x_\lambda)$,
we need to 
define a function
%(see Lemma \ref{lem:der-sx} and Lemma \ref{lem:der-dx}),
$G\colon \mathbb{R}^+ \to \mathbb{R}$ as follows:
\begin{equation*}
	\label{eq:def-G}
	G(\mu) = 
	\int_{0}^{\pi}\frac{\text{d}\sigma}{\sqrt{\mu + 4\pi\sin\sigma}}.
\end{equation*}
\begin{remark}
	The function $G$ is strictly decreasing,
	$G(0) > 1$ and $\lim_{\mu \to \infty}G(\mu) = 0$.
	As a consequence, there exists an unique $E > 0$
	such that 
	\begin{equation}
		\label{eq:cond-E}
		\int_{0}^{\pi}\frac{\text{d}\sigma}{\sqrt{E + 4\pi\sin\sigma}} = 1.
	\end{equation}
\end{remark}

\begin{remark}
	\label{rem:G-properties}
	In both Lemma \ref{lem:der-sx} and Lemma \ref{lem:der-dx},
	we will exploit the following limit,
	obtained  from the definition of 
	$\phi^*_\lambda$ given in \eqref{eq:def_phiStar}:
	\begin{equation}
		\label{eq:lim_lambdax02}
		\lim_{\lambda \to \infty} \lambda x_\lambda^2 = 2\pi.
	\end{equation}
\end{remark}

\begin{lemma}
	\label{lem:der-sx}
	Let $\ell_\lambda \in \mathfrak{L}_\lambda$
	be the minimizer of $L_\lambda$.
	Let $E > 0$ be such that \eqref{eq:cond-E} holds.
	Then 
	\begin{equation*}
		\lim_{\lambda \to \infty} x_\lambda \ell_\lambda'(x_\lambda)
		=  - E.
	\end{equation*}
\end{lemma}
\begin{proof}
	By the change of variable $x = x_\lambda t$,
	for every $\lambda$ we can reparametrize 
	the integral functional $L_\lambda$ defined in
	\eqref{eq:def-L} on the interval $[0,1]$.
	Hence, setting
	\[
		\varphi(t) = \phi(x_\lambda t) = \phi(x),
	\]
	we obtain
	\begin{equation}
		\label{eq:F1-01}
		L_\lambda(\phi) = 
		\frac{1}{x_\lambda}
		\int_{0}^{1}\left(
			\frac{1}{2}(\varphi')^2
			- \lambda x_\lambda^2 (1 - x_\lambda t)\sin\varphi
		\right)\text{d}t.
	\end{equation}
	As a consequence, the function
	$\widetilde\ell_\lambda\colon [0,1] \to \mathbb{R}$,
	defined as
	$\widetilde\ell_\lambda(t) = \ell_\lambda(x_\lambda t)$,
	minimizes the integral
	\[
		\int_{0}^{1}\left(
			\frac{1}{2}(\varphi')^2
			- \lambda x_\lambda^2 (1 - x_\lambda t)\sin\varphi
		\right)\text{d}t,
	\]
	and it never equals the function
	$\varphi^*_\lambda(t) = \phi^*_\lambda(x_\lambda t)$.
	Therefore, it is a solution of the following Dirichlet problem
	\[
		\begin{cases}
			y'' + \lambda x_\lambda^2(1- x_\lambda t)\cos y = 0,\\
			y(0)= 0, \quad y(1) = -\pi.
		\end{cases}
	\]
	Arguing similarly as in the proof of Proposition \ref{prop:convergence}
	and recalling \eqref{eq:lim_lambdax02},
	$\widetilde\ell_\lambda$ converges in the $C^1$ norm to the minimizer
	$\widetilde{z}$ of the functional 
	\[
		L(z) =
		\int_{0}^{1}\left(
			\frac{1}{2}(z')^2
			-2\pi \sin z
		\right)\text{d}t,
	\]
	subject to 
	\[
		z(t)\le \lim_{\lambda \to \infty}\varphi^*_\lambda(t)
		= \lim_{\lambda\to \infty}
		\left[
			\frac{\lambda x_{\lambda}^2}{2}t^2
			\left(\frac{x_\lambda}{3}t - 1 \right) - \frac{x_\lambda^2}{2}t^2
		\right]
		= -\pi t^2,
		\quad \forall t \in [0,1].
	\]
	Therefore, $\widetilde{z}$ is a solution of the
	following Dirichlet problem
	\[
		\begin{cases}
			z'' + 2 \pi \cos z = 0,\\
			z(0)= 0, \quad z(1) = -\pi,
		\end{cases}
	\]
	and 
	\begin{equation}
		\label{eq:z1-velo-repar}
		\lim_{\lambda \to \infty} x_\lambda \ell_\lambda'(x_\lambda^-)
		= \lim_{\lambda \to \infty} (\widetilde\ell_\lambda)'(1)
		= \widetilde{z}'(1).
	\end{equation}
	Being the minimizer of a suitably regular autonomous problem, $\widetilde{z}$ is a monotone function
	(see, e.g., \cite[Theorem 3.1]{cupini2007}),
	so we have that $\widetilde{z}'(t)\le 0$ for all $t \in [0,1]$.
	Since $z'' + 2\pi\cos z = 0$ is an autonomous differential equation,
	there exists $E \in \mathbb{R}$ such that
	\begin{equation}
		\label{eq:z1-constE}
		\frac{1}{2}(\widetilde{z}')^2 +  2 \pi \sin \widetilde{z}
		=\frac{1}{2} E, \quad \forall t \in [0,1].
	\end{equation}
	Moreover, being $\widetilde{z}(1) = -\pi$,
	we obtain that $(\widetilde{z}'(1))^{2} = E$.
	Since $\widetilde{z}'(t) \le 0$,
	we have $\widetilde{z}'(1) = -\sqrt{E}$ and
	from \eqref{eq:z1-constE} we obtain
	\[
		\frac{\widetilde{z}'}{\sqrt{E - 4\pi \sin \widetilde{z}}} = -1.
	\]
	Recalling that $\widetilde{z}(0) = 0$ and $\widetilde{z}(1) = -\pi$,
	with some easy computations we obtain
	\begin{equation*}
		\int_{0}^{\pi}\frac{\text{d}\sigma}{\sqrt{E + 4\pi\sin\sigma}} = 1,
	\end{equation*}
	hence $E$ satisfies \eqref{eq:cond-E}.
	By Remark \ref{rem:G-properties},
	there exists a unique $E > 0$ which satisfies \eqref{eq:cond-E}
	and using also
	\eqref{eq:z1-velo-repar}
	we have
	\[
		\lim_{\lambda \to \infty} x_\lambda\ell_\lambda'(x_\lambda^-)
		= - \sqrt{E}.
	\]
\end{proof}

\begin{lemma}
	\label{lem:der-dx}
	Let $r_\lambda \in \mathfrak{R}_\lambda$
	be the minimizer of $R_\lambda$.
	Then 
	\begin{equation*}
		\lim_{\lambda \to \infty} x_\lambda r'(x_\lambda)
		=  - 2\sqrt{\pi}.
	\end{equation*}
\end{lemma}
\begin{proof}
	The proof is similar to the one of Lemma \ref{lem:der-sx}.
	By the change of variable $x = x_\lambda t$, 
	we re-parameterize the integral functional $R_\lambda$ defined in \eqref{eq:def-R}
	on the interval $[1,1/x_\lambda]$.
	As a consequence, the function 
	$\widetilde{r}_\lambda\colon [1,1/x_\lambda] \to \mathbb{R}$
	given by $\widetilde{r}_\lambda(t) = r_\lambda(x_\lambda t)$
	is a minimizer for the functional 
	\[
		\varphi \mapsto
		\int_{0}^{1}\left(
			\frac{1}{2}(\varphi')^2
			- \lambda x_\lambda^2 (1 - x_\lambda t)\sin\varphi
		\right)\text{d}t,
	\]
	with the constraint $\widetilde{r}_\lambda(1) = -\pi$ and
	$\widetilde{r}_\lambda(t) \le -\pi$.
	Let us notice that
	\begin{equation}
		\label{eq:r-rescale-vel}
		x_\lambda r_\lambda'(x_\lambda)
		= \widetilde{r}_\lambda(1).
	\end{equation}
	Using arguments similar to Lemma \ref{lem:tildephi_ge_32pi} and
	Lemma \ref{lem:noTouch-after-xlambda},
	we obtain that $\widetilde{r}_\lambda(t) \in(-\frac{3}{2} \pi, -\pi)$
	for all $t \in (1,1/x_\lambda]$.
	Therefore, it is the solution of the following Cauchy problem
	\begin{equation}
		\label{eq:Cauchy-zlambda}
		\begin{cases}
			z'' + \lambda x_\lambda^2 (1 - x_\lambda t) \cos z = 0,\\
			z(1)= -\pi,\\
			z'(1) = -\nu_\lambda,
		\end{cases}
	\end{equation}
	for a suitable $\nu_\lambda > 0$.
	Let us show that $\nu_\lambda > 0$ is upper bounded.
	For every $\lambda $ we have
	\[
		z'(t) + \nu_\lambda = \int_{1}^{t}z''(s)\text{d}s.
	\]
	If $\lambda$ is sufficiently large, $1/x_\lambda > 2$ and we 
	can integrate the both sides of the previous equation
	on the interval 
	$[1,2]$.
	Recalling that $z(1) = - \pi$ and that $z(2) > -\frac{3}{2}\pi$, we obtain
	\[
		\begin{multlined}
			\nu_\lambda
			= - \int_1^2 z'(t)\text{d}t  + \int_{1}^{2}\int_{1}^{t}z''(s)\text{d}s\text{d}t \\
			=z(1) - z(2)
			- \int_{1}^{2}\int_{1}^{t}\lambda x^2_\lambda(1 - x_\lambda s)\cos z\  \text{d}s
			\text{d}t
			\le \frac{\pi +  \lambda x_\lambda^2}{2}.
		\end{multlined}
	\]
	By \eqref{eq:lim_lambdax02}, if $\lambda$ is sufficiently large we have
	\[
		\nu_\lambda \le \frac{\pi +  2\pi}{2} + \frac{\pi}{2} = 2\pi.
	\]
	Therefore,
	there exists $\nu > 0$ such that,
	up to subsequences,
	$\nu_\lambda \to \nu$ as $\lambda \to \infty$
	and, on every compact set $[1,M]$,
	the solutions of \eqref{eq:Cauchy-zlambda} uniformly converges to 
	the solution $w\colon [1,+\infty) \to \mathbb{R}$ of the following Cauchy problem
	\[
		\begin{cases}
			z'' + 2\pi \cos z = 0,\\
			z(1)= 0,\\
			z'(1) = -\nu.
		\end{cases}
	\]
	By the minimality conditions on $\widetilde{r}_\lambda$,
	in particular by $\widetilde{r}'_\lambda(1/x_\lambda) = 0$,
	we obtain that 
	$\nu > 0$ is such that 
	$\lim_{t \to \infty}w(t) = - \frac{3}{2}\pi$
	and
	$\lim_{t \to \infty}w'(t) = 0$.
	Since $z'' + 2\pi \cos z = 0$
	is an autonomous differential equation,
	there exists a constant $E > 0$ such that
	\[
		\frac{1}{2}(w')^2 + 2\pi\sin w = E,
	\]
	and since $\sin w(1) = 0$ we obtain
	\[
		E = \frac{1}{2}\nu^2.
	\]
	Moreover, since $\lim_{t \to \infty}w(t) = -\frac{3}{2}\pi$
	and $\lim_{t \to \infty}w'(t) = 0$,
	we have that 
	\[
		E = \frac{1}{2}\nu^2
		= \lim_{t \to \infty} \frac{1}{2}(w'(t))^2+ 2\pi\sin w'(t)
		= 2\pi,
	\]
	hence $\nu^2 = 4 \pi$.
	Thus, using also \eqref{eq:r-rescale-vel},
	we obtain 
	\[
		\lim_{\lambda \to \infty}x_\lambda r'(x_\lambda)
		= \lim_{\lambda \to \infty} - \nu_\lambda 
		= - \nu = -2\sqrt{\pi}.
	\]
\end{proof}

\subsection{Existence of local minimizers distinct form the global ones}
\label{sub:33}

\begin{lemma}
	\label{lemma_N}
	Let $(\widetilde\phi,\widetilde\theta)$ be a global minimizer
	of $F_{b,k}$ in $\mathfrak{S}^*_\lambda$,
	whose existence is ensured by Proposition \ref{prop:P*existence}.
	Then, for sufficiently small $b$ and sufficiently large $\lambda$, we have
	\[
		\widetilde{\phi}(x)<\phi_\lambda^*(x),
		\quad \forall x \ne 0.
	\]
\end{lemma}

\begin{proof}
	By Corollary \ref{cor:notTouch}, we have that 
	$\widetilde\phi$ can be equal to $\phi^*_\lambda$ 
	only in $0$ and $x_\lambda$.
	So our aim is to prove that if $b$ and $k/b$ are sufficiently small
	(or equivalently if $b$ is sufficiently small and $\lambda$ is sufficiently large),
	then 
	\[
		\widetilde\phi(x_\lambda) < -\pi.
	\]
	As a first step,
	let us show that if $\lambda = b/k$ is sufficiently large,
	then $\widetilde\phi_\lambda(x_\lambda) < -\pi$,
	where $\widetilde\phi_\lambda$ is the minimizer of the functional 
	$F_\lambda$ defined in \eqref{eq:def-Flambda}.
	By Lemma \ref{lem:phiTouch}
	and the definitions of $\ell_\lambda$ and $r_\lambda$,
	we need to prove that for $\lambda$ sufficiently large we have
	\[
		\ell_\lambda'(x_\lambda) > r_\lambda'(x_\lambda),
	\]
	namely we need to prove that \eqref{eq:opt-cond-lr} does not hold.
	Using Lemma \ref{lem:der-sx} and 
	Lemma \ref{lem:der-dx}, it suffices to prove that 
	\[
		-\sqrt{E} > -  2\sqrt{\pi},
	\]
	or, equivalently, that $4\pi > E$.
	Since $E$ satisfies \eqref{eq:cond-E} and $G$ is a decreasing function,
	we need to show that 
	\[
		1 = G(E) > G(4\pi).
	\]
	Since the $sin(\sigma)> 0$ for every $\sigma \in (0,\pi)$,
	we have
	\[
		G(4\pi)
		= \int_{0}^{\pi}\frac{\text{d}\sigma}{\sqrt{4\pi + 4\pi\sin\sigma}}
		< \frac {\sqrt{\pi}}{2} < 1,
	\]
	hence we infer that if $\lambda$ is sufficiently large,
	then $\widetilde\phi_\lambda(x_\lambda) < -\pi$.

	\noindent Therefore, let us fix such a $\lambda$.
	By contradiction, let
	$(b_n,k_n)_{n \in \mathbb{N}} \subset \mathbb{R}^+ \times \mathbb{R}^+$
	be a sequence such that $b_n \to 0$, $b_n / k_n = \lambda$ for all $n$
	and the sequence of minimizers of $F_{b_n,k_n}$ in
	$\mathfrak{S}^*_\lambda$, denoted by 
	$(\widetilde\phi_n,\widetilde\theta_n)$, is such that
	\[
		\widetilde\phi_n(x_\lambda) = - \pi, \quad \forall n \in \mathbb{N}.
	\]
	By Proposition \ref{prop:convergence},
	there exists a subsequence of $\widetilde\phi_n$
	that converges uniformly to $\widetilde\phi_\lambda$.
	Since $\widetilde\phi_\lambda(x_\lambda) < -\pi$,
	this is absurd, and we are done.

\end{proof}

\begin{remark}
	Lemma \ref{lemma_N} is not enough to conclude that
	$(\widetilde{\phi},\widetilde{\theta})$
	is a local minimizer of the functional $F$ in $\mathfrak{S}$,
	since $\widetilde{\phi}$ belongs to $\partial \mathfrak{C}^*_\lambda$.
	Indeed, consider the sequence $(f_k)_k\subset \mathfrak{C}^*$ defined as follows:
	\[
		f_k=
		\begin{cases}
			kx, &		\mbox{for } x\in [0,k^{-3}], \\
			k^{-2}, &		\mbox{for }  x\in (k^{-3},1].
		\end{cases}
	\]
	Clearly we have $\lVert f_k\rVert_{H^1}\rightarrow 0$ when $k\rightarrow \infty$. 
	For every $C$-Lipschitz $\phi\in \mathfrak{C}^*_\lambda$, $\phi(x)+f_{k}(x)>0$ for $x\in [0,k^{-3}]$ if $k>C$. 
	Therefore every Lipschitz-regular element of $\mathfrak{C}^*$
	belongs to $\partial \mathfrak{C}^*$, hence
	$\tilde\phi \in \partial\mathfrak{C}^*$.
\end{remark}

\noindent As a consequence of the previous remark,
we need to show that there is a sufficiently small open ball
$\widetilde{\mathcal{B}}$
centered in $\widetilde{\phi}$ such that
$F(\widetilde{\phi},\widetilde{\theta})\le F(\phi,\theta)$
for every $(\phi,\theta)\in \widetilde{\mathcal{B}}\times L^2$.
We obtain this result by a $\Gamma$--convergence argument,
for which we need the following definitions.

\begin{definition}
	Fix $b,k > 0$ and set $\lambda = b/k$.
	For each real number $\epsilon > 0$
	we define the function 
	$\phi^*_{\lambda,\epsilon}\colon [0,1] \to \mathbb{R}$ as follows:
	\begin{equation*}
		\phi_{\lambda,\epsilon}^*(x)=
		\max\left\{
			\frac{\lambda}{2}x^2 \left(\frac{x}{3} -1\right)
			- \frac{1}{2}x^2 + \epsilon,
			- \pi
		\right\}.
	\end{equation*}
	We denote by $x_{\lambda,\epsilon}$ the least $x$
	such that $\phi_{\lambda,\epsilon}^* = -\pi$,
	that is:
	\[
		x_{\lambda,\epsilon}=
		\min\left\{x \in [0,1]: \phi^*_n(x) = -\pi\right\},
	\]
	and we define
	\begin{equation*}
		\mathfrak{C}_{\epsilon}^* \coloneqq
		\left\{
			\phi \in H^{1}_{*}: \phi(x) \le \phi_{\lambda,\epsilon}^*(x),
			\ \forall x \in [0,1]
		\right\}.
	\end{equation*}
	Recalling the definition of $\Theta\colon H^1_* \to L^2$
	given in Corollary \ref{cor:Theta},
	we define the functional
	$
	\mathcal{F}_{\epsilon}\colon H^1_* \to \mathbb{R}
	$
	as follows:
	\[
		\mathcal{F}_{\epsilon}(\phi) =
		\begin{cases}
			F_{b,k}(\phi,\Theta(\phi)),
				&		\mbox{if } \phi \in \mathfrak{C}^*_\epsilon, \\
			+\infty, &		\mbox{otherwise},
		\end{cases}
	\]
	and the functional $\mathcal{F}\colon H^1_*([0,1],\mathbb{R}) \to \mathbb{R}$ as follows:
	\begin{equation*}
		\mathcal{F}(\phi) = 
		\begin{cases}
			F_{b,k}(\phi,\Theta(\phi)), &		\mbox{if } \phi \in \mathfrak{C}^*, \\
			+\infty, &		\mbox{otherwise}.
		\end{cases}
	\end{equation*}
\end{definition}
\begin{remark}
	If $b < 1$ and $(\widetilde\phi,\widetilde\theta) \in \mathfrak{S}^*$
	is a global minimizer for $F_{b,k}$,
	then $\widetilde\theta = \Theta(\widetilde\phi)$ a.e.,
	thus
	\[
		F_{b,k}(\widetilde\phi,\widetilde\theta) = \mathcal{F}(\widetilde\phi).
	\]
\end{remark}

\begin{lemma}
	\label{lem:minFn}
	For every $b,k,\epsilon > 0$,
	$\mathcal{F}_\epsilon\colon H^1_*\to \mathbb{R}$
	admits a global minimizer,
	that we denote by $\widetilde\phi_\epsilon$.
	Moreover,
	\begin{equation}
		\label{eq:phi_e-noTouch}
		\widetilde\phi_\epsilon(x) 
		< \phi^*_{\lambda,\epsilon}(x),
		\quad\forall x \ne x_{\lambda,\epsilon}.
	\end{equation}
\end{lemma}
\begin{proof}
	Since $ \mathfrak{C}_{\epsilon}^* \subset H^1_*$
	is closed with respect to the $L^\infty$ norm,
	the existence of a global minimizer
	can be obtained following the same proof of 
	Proposition \ref{prop:P*existence}.
	Moreover, as previously observed in Remark \ref{rem:onlySecondDeriv},
	the proof of Corollary \ref{cor:notTouch}
	relies only on the second derivative of $\phi^*$,
	hence the analogous result given in
	\eqref{eq:phi_e-noTouch}
	holds when we substitute $\phi^*$ with $\phi^*_\epsilon$.
\end{proof}

\begin{lemma}
	\label{gammaconv}
	Let $(\epsilon_n)_n\subset \mathbb{R}$ be a strictly decreasing
	sequence such that
	$\epsilon_n > 0$ for each $n$ and $\epsilon_n \to 0$.
	Then the sequence of functionals $\mathcal{F}_{\epsilon_n}$
	$\Gamma$-converges to the functional $\mathcal{F}$.
	Therefore,
	if $(\phi_{\epsilon_n})_n\subset H^1_*$
	is a sequence of absolute minimizers of $\mathcal{F}_{\epsilon_n}$,
	it converges in $L^\infty$-norm to a minimizer of $\mathcal{F}$.
\end{lemma}
\begin{proof}
	Since $\epsilon_n$ is strictly decreasing,
	for every $n \in \mathbb{N}$
	we have
	$\mathfrak{C}^*_{\lambda}\subset \mathfrak{C}^{*}_{\lambda,\epsilon_{n+1}}
	\subset \mathfrak{C}^{*}_{\epsilon_n}$.
	Therefore, 
	the sequence $\mathcal{F}_n$ is pointwise non-decreasing,
	so that
	(see e.g. Remark 1.40 in \cite{braides})
	\[
		\Gamma\text{-}\lim_n \mathcal{F}_n
		=\sup_n \text{sc}(\mathcal{F}_n)
		=\lim_n \text{sc}(\mathcal{F}_n)
		=\lim_n\mathcal{F}_n
		=\mathcal{F},
	\]
	where $\text{sc}(\cdot)$ indicates the lower-semicontinuous
	envelope and the penultimate equality holds because $\mathcal{F}_n$
	is weakly lower-semicontinuous for every $n$. 
	The second part of the statement follows from the basic properties of $\Gamma$-convergence
	(see Theorem 1.21 in \cite{braides}).
\end{proof}

\noindent We are finally ready to prove our main result.
\begin{proof}[\textbf{Proof of Theorem \ref{teo:local-min}}]
	By Lemma \ref{lemma_N},
	if $b$ is sufficiently small
	and $\lambda$ is sufficiently large,
	any minimizer of $F_{b,k}$ in $\mathfrak{S}^*_{\lambda}$,
	denoted by 
	$(\widetilde\phi,\widetilde\theta)$,
	is such that $\widetilde\phi(x) < \phi^*_\lambda(x) < 0$ for every
	$x \ne 0$.
	To prove Theorem \ref{teo:local-min},
	we need to prove that 
	$(\widetilde\phi,\widetilde\theta)$ is a local minimizer
	in the whole set $\mathfrak{S} = H^1_*\times L^2$.
	It is important to notice that we can choose $\lambda$
	large enough such that
	$\widetilde\phi(x_\lambda) < \phi^*_\lambda(x_\lambda) = -\pi$
	and that 
	we can assume that $b < 1$,
	so that Corollary \ref{cor:Theta} 
	can be applied to define the function $\Theta$.

	\noindent Seeking a contradiction,
	fix $\alpha > 0$ and 
	let $\widetilde{\mathcal{B}}_{\alpha} \subset H^1_*$
	be the open ball with center $\widetilde\phi$ and radius $\alpha$
	with respect to the $H^{1}$ norm. 
	Then, by contradiction, there exists
	$(\phi_\alpha,\theta_\alpha) \in \widetilde{\mathcal{B}}_\alpha\times L^2$
	such that 
	\[
		F_{b,k}(\phi_\alpha,\theta_\alpha) < F_{b,k}(\widetilde\phi,\widetilde\theta).
	\]
	By definition of $\Theta$, this implies
	\[
		\mathcal{F}(\phi_\alpha) \le 
		F_{b,k}(\phi_\alpha,\theta_\alpha) 
		< F_{b,k}(\widetilde\phi,\widetilde\theta)
		= \mathcal{F}(\widetilde\phi).
	\]
	Since 
	$(\widetilde\phi,\widetilde\theta)$ is a global minimizer of $F_{b,k}$
	in $\mathfrak{S}^* = \mathfrak{C}^{*}\times L^2$,
	the previous chain of inequalities implies that $\phi_\alpha\notin \mathfrak{C}^*$.
	Let $\beta > 0$ such that 
	$\phi_\alpha \in \mathfrak{C}^*_{\beta}$.
	By Lemma \ref{lem:minFn},
	there exists a global minimizer of $\mathcal{F}_{\beta}$,
	that we denote by $\widetilde\phi_\beta$,
	and we have
	\[
		\mathcal{F}_{\beta}(\widetilde\phi_\beta)
		\le \mathcal{F}_\beta(\phi_\alpha)
		< \mathcal{F}(\widetilde\phi).
	\]
	Hence, $\widetilde\phi_\beta \in \mathfrak{C}^*_\beta\setminus\mathfrak{C}^*$
	and there exists
	$\epsilon \in ]0,\beta]$ such that
	$\widetilde\phi_\beta\in \mathfrak{C}^*_\epsilon$
	and 
	\[
		\left\{
			x \in [0,1]: \widetilde\phi_\beta(x) = \phi^*_{\lambda,\epsilon}(x)
		\right\} \ne \emptyset.
	\]
	Since $\mathfrak{C}^*_\epsilon \subseteq \mathfrak{C}^*_\beta$, we have that we can take $\widetilde\phi_\epsilon=\widetilde\phi_\beta$, that is $\widetilde\phi_\beta$ is actually a global minimizer of $\mathcal{F}_\epsilon$.
	As a consequence, applying again Lemma \ref{lem:minFn},
	we have that
	$\widetilde\phi_\epsilon(x_{\lambda,\epsilon}) 
	= \phi^*_{\lambda,\epsilon}(x_{\lambda,\epsilon}) = -\pi$.
	For the sake of presentation,
	an illustration of the above construction is given in 
	Fig.~\ref{fig:finalproof}.
	\begin{figure}[h]
		\centering
		\includegraphics[width=0.7\linewidth]{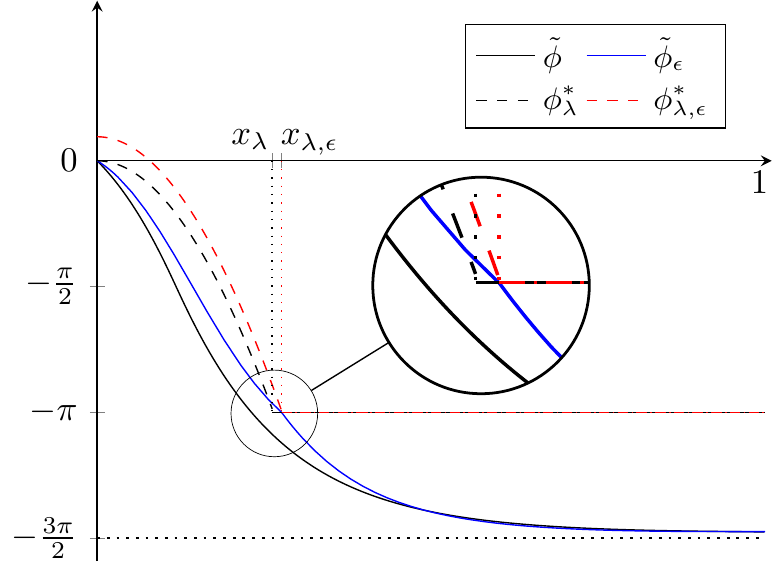} 
		\caption{
			Convergence of minimizers in the $\Gamma$--convergence
			argument for the proof of Theorem \ref{teo:local-min}.
		}
		\label{fig:finalproof}
	\end{figure}

	\noindent Now, consider a strictly decreasing sequence $(\alpha_n)_n \subset \mathbb{R}^+$
	that converges to $0$.
	Applying the previous construction,
	we can construct a strictly decreasing sequence
	$(\epsilon_n)_n \subset \mathbb{R}^+$
	that converges to $0$ and a sequence
	$(\widetilde\phi_{\epsilon_n})_n \subset  H^1_*$
	of global minimizers of $\mathcal{F}_{\epsilon_n}$
	 such that 
	\[
		\widetilde\phi_{\epsilon_n}(x_{\lambda,\epsilon_n}) = -\pi.
	\]
	By Lemma \ref{gammaconv},
	the sequence 
	$(\widetilde\phi_{\epsilon_n})_n $
	uniformly converges to 
	a minimizer $\bar\phi$ of $\mathcal{F}$.
	Since $ \lim_{n\to \infty} x_{\lambda,\epsilon_n} = x_\lambda $
	we have
	\[
		\bar\phi(x_\lambda)
		= \lim_{n\to \infty} \widetilde\phi_{\epsilon_n}(x_{\lambda,\epsilon_n})
		= - \pi.
	\]
	By Lemma \ref{lemma_N}, this is a contradiction 
	and, finally, we are done.
\end{proof}

\subsection{Regularity of local minimizers}

By Proposition \ref{prop:regularity_global_min},
we obtain the following regularity result for the local minimizers of $F_{b,k}$
such that $\phi(x) < 0$ in $(0,1]$,
whose existence is ensured by Theorem \ref{teo:local-min}.
\begin{cor}
	\label{cor:regularity_local_min}
	Let $(\widetilde{\phi},\widetilde{\theta})$ be a local minimizer of $F_{b,k}(\phi,\theta)$ 
	such that $\phi(x) < 0$ in $(0,1]$.
	If $b \le 1$,
	then both $\widetilde\phi$ and $\widetilde\theta$ are $C^{\infty}$.
\end{cor}

\noindent A complementary result can be established assuming a lower bound on $b$
(depending on $x$).
This is obtained in Proposition \ref{thetajumps},
after we establish a lemma showing that $|\widetilde{\phi}-\widetilde{\theta}|$
cannot exceed $\pi$. Proposition \ref{thetajumps} entails that, in general,
there is no hope to be able to ``neglect''
the asymmetric nature of the variational problem,
since the less regular component ($\theta$)
of the local minimizers can actually live in $L^2\setminus H^1$. 

\begin{lemma}
	Let $(\widetilde{\phi},\widetilde{\theta})$ be a local minimizer of $F_{b,k}(\phi,\theta)$.  Let $\left[(2n-1)\frac{\pi}{2},(2n+1)\frac{\pi}{2}\right]:=I_n$ for $ (n\in \mathbb{Z})$. If $\widetilde{\phi}(\bar{x})\in (I_n)^{\circ}$ then $\widetilde{\theta}(x)\in I_n$ a.e. on a neighborhood of $\bar{x}$.
	\label{theta_follows}
\end{lemma}
\begin{proof}
	First of all notice that, for every measurable subset $M\subset [0,1]$, we have that $(\widetilde{\phi},\widetilde{\theta})$ can be a (global or local) minimizer for $F_{b,k}$ only if 
	\begin{equation}
		\label{condition_theta}
		\widetilde{\theta}=\min_{\theta} \left(\frac{(\widetilde{\phi}-\theta)^2}{2}-b(1-x)\sin{\theta}\right)
	\end{equation} 
	a.e. on $M$.
	Since the function $\sin(x\pm (2n\pm 1)\frac \pi2)$ is an even function, it follows that, if $\widetilde{\phi}(\bar{x})\in (I_n)^{\circ}$, the minimization of the term $\frac{(\phi-\theta)^2}{2}$ implies $\widetilde{\theta}(x)\in I_n$ a.e. on a neighborhood of $\bar{x}$. Supposing indeed $\widetilde{\theta}(x)\in I_{n\pm 1}$, one can replace $\widetilde{\theta}(x)$ by the symmetric value with respect to $(2n-1)\frac \pi2$ (if $\widetilde{\phi}$ belongs to the left half of $I_n$) or with respect to $(2n+1)\frac \pi 2$ (if $\widetilde{\phi}$ belongs to the right half of $I_n$), obtaining the same value for the term $b(1-x)\sin{\theta}$ and a strictly smaller value for the term $\frac{(\widetilde{\phi}-\theta)^2}{2}$.
\end{proof}
\begin{prop}
	\label{thetajumps}
	Suppose that there exists a local minimizer $(\widetilde{\phi} , \widetilde{\theta})$ of $F(\phi,\theta)$ in $\mathfrak{S}$ such that $\widetilde{\phi}(x)<-\frac{\pi}{2}$ for some $x\in (0,1)$ such that $b>\frac{1}{1-x}$.
	Then $\widetilde{\theta}\notin C^0$.  
\end{prop}
\begin{proof}

	The function
	$\widetilde{\theta}$ must solve the localized problem 
	\[
		\inf_{\theta}\int_{S_{0}}
		\left[\frac{k}{2}(\widetilde{\phi}')^2
		+ \frac{(\widetilde{\phi}-\theta)^2}{2}-b(1-x)\sin\theta \right]\text{d}x,
	\]
	where $S_{0}$ is any maximal sub-interval of $[0,1]$ such that
	$\widetilde{\phi}(x)\in (I_{n})^{\circ}$ for $x\in S_{0}$.
	Therefore $\widetilde{\theta}(x)\in I_{0}$ a.e. in $S_{0}$.

	\noindent On the other hand, $\widetilde{\theta}$ also solves the localized problem
	\[
		\inf_{\theta}\int_{S_{-1}}
		\left[
			\frac{k}{2}(\widetilde{\phi}')^2
			+ \frac{(\widetilde{\phi}-\theta)^2}{2}-b(1-x)\sin\theta
		\right]dx
	\]
	where is any maximal sub-interval of $[0,1]$
	such that $\widetilde{\phi}(x)\in (I_{-1})^{\circ}$ for $x\in S_{-1}$.
	Therefore $\widetilde\theta(x)\in I_{-1}$ a.e. on $ S_{-1}$.

	\noindent Since $\widetilde{\phi}(0)=0$ and $\widetilde{\phi}(x)<-\frac{\pi}{2}$
	for some $x\in(0,1)$,
	there exist two nonempty such intervals $S_{0}$ and $S_{-1}$.
	The continuity of $\widetilde{\phi}$ implies that $\widetilde{\theta}$
	can be continuous only if $\widetilde{\theta}(x)=-\frac{\pi}{2}$
	at those $x$ such that $\widetilde{\phi}(x)=-\frac{\pi}{2}$.
	However we have:
	\begin{align*}
		&\frac{\partial}{\partial \theta}
		\left(\frac{\theta^2}{2}-\phi\theta-b(1-x)\sin{\theta}\right)
		\bigg|_{\theta=\phi=-\frac{\pi}{2}}=0,\\
		&\frac{\partial^2 }{\partial \theta^2} 
		\left(\frac{\theta^2}{2}-\phi\theta-b(1-x)\sin{\theta}\right)
		\bigg|_{\theta=-\frac{\pi}{2}}= 1-b(1-x)<0,
	\end{align*}
	and therefore $\widetilde{\theta}$ cannot verify \eqref{condition_theta}.
	This contradiction implies that $\widetilde\theta(x)\ne -\frac{\pi}{2}$,
	so $\widetilde{\theta}\notin C^0$.
\end{proof}

\section{Further questions}

The results achieved in this paper open some new questions.
The limit processes employed during the proof of Theorem \ref{teo:local-min}
do not allow to estimate the lower bound for $\lambda$ and the upper bound for
$b$ that ensure the existence of
a local minimizer $(\widetilde\phi,\widetilde\theta)$ of $F_{b,k}$
such that $\widetilde\phi(x) \le 0$.
By Corollary \ref{cor:notTouch} and Proposition \ref{prop:convergence},
the upper bound for $b$ depends on $\lambda$: in particular,
it depends on the distance between $\widetilde\phi_\lambda$,
namely the minimizer of $F_\lambda$ in $\mathfrak{C}^*_\lambda$,
and $-\pi$ at $x_\lambda$.
Even if giving such estimations is still an open problem,
some numerical simulations conducted by the authors
suggest that if $\lambda$ is greater than $42$, then 
the minimizer of $F_\lambda$ does not ``touch" $\phi^*_\lambda$,
as it can be seen in Fig.~\ref{fig:42notouch}.
\begin{figure}[h]
	\begin{subfigure}{0.48\textwidth}
		\includegraphics[width = 0.9\textwidth]{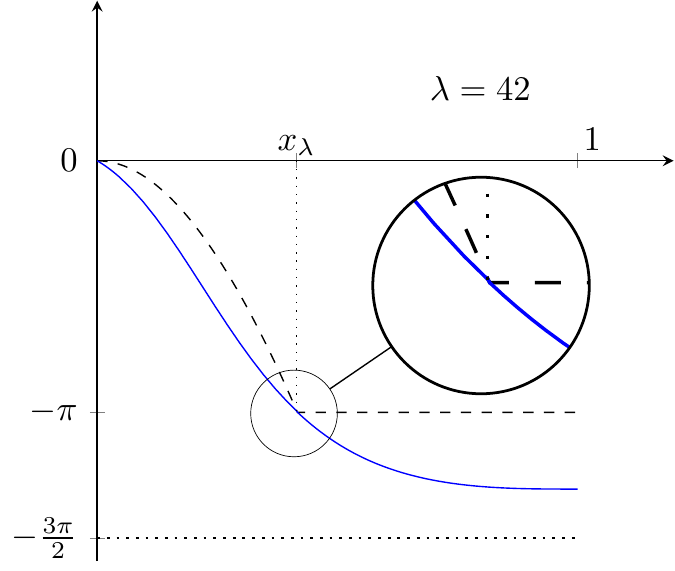}
	\end{subfigure}
	\hfill
	\begin{subfigure}{0.48\textwidth}
		\includegraphics[width = 0.9\textwidth]{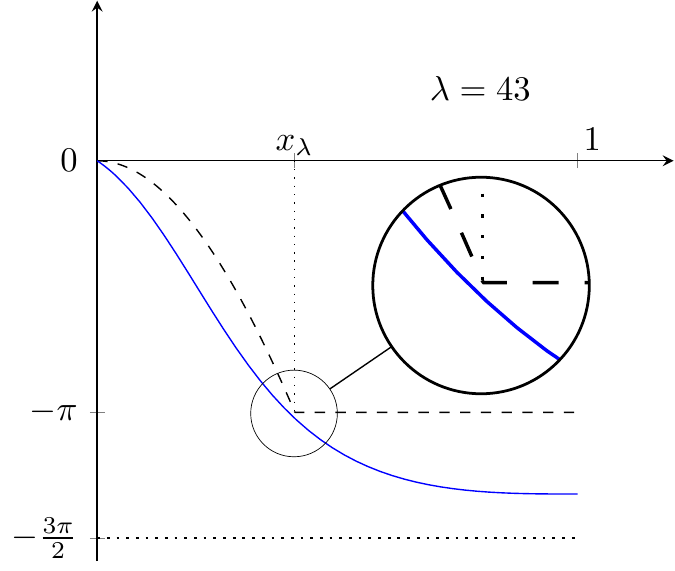}
	\end{subfigure}
	\caption{
		Two minimizers of $F_\lambda$ in $\mathfrak{C}^*_\lambda$:
		if $\lambda = 42$, then the minimizer passes through $(x_\lambda,-\pi)$
		and it is not a local minimizer of $F_\lambda$ in $H^1_*([0,1])$;
		if $\lambda = 43$, then the minimizer does not meet $\phi^*_\lambda$
		and it is indeed a local minimizer of $F_\lambda$ in $H^1_*([0,1])$.
	}
	\label{fig:42notouch}
\end{figure}

\noindent It is natural to try to generalize
the results developed herein in various directions.

\noindent Firstly, from the point of view of calculus of variations,
it would be interesting to investigate what happens to local minimizers when
the potential $V$ has a more general form,
rather than the form $b(1-x)\sin\theta$;
this could be interesting for the application to one-dimensional continua with more complicated microstructures than the one considered in the Timoshenko beam,
as for instance the ones investigated in \cite{Barch,Berez,Placidi,Turco}.

\noindent Secondly,
from the point of view of elasticity theory,
the generalization of the inextensibility constraint \eqref{inextensible} leads,
in its simplest form, to a further additive term in the integrand of type
$C(\|\boldsymbol{\chi}'\|-1)^2/2$ and to a potential of the form
\[
	V(x,\theta)=b\|\boldsymbol{\chi}'\|(1-x)\sin\theta,
\]
and thus it also introduces new problems.

\noindent Finally,
it is also natural to try to generalize the existence and regularity results
concerning the global minimizer (developed in Section \ref{sec:global-min})
to problems living
in $W^{m,p}\times W^{n,p}$ of type: 
\[
	\inf_{u,v}\int_{\Omega}\bigg(
	f(\nabla^m u )+g(\nabla^n v)+h(u-v)-V(x,u,v)\bigg)\text{d}x,
\]
where $\Omega$ is a bounded domain of an Euclidean space
and $m$ is strictly larger than $n$. One may expect that,
assuming $f,g,h$ nice enough and suitable boundary conditions,
the term in $u-v$ should allow to gain $W^{m,p}$
regularity for both elements of the minimizing pair $(\bar{u},\bar{v})$.

%------
% Insert acknowledgments and information
% regarding funding at the end of the last
% section, i.e., right before the bibliography.
%------

\section{Acknowledgments}
	The authors thank Pierre Seppecher and Roberto Giambò for useful discussions on the subject.

%------
% Insert the bibliography.
%------

%\begin{thebibliography}{99}

%------ Example for a paper in journal:
% \bibitem{article1}
% A.~Petrunin, Parallel transportation for Alexandrov space with curvature bounded below.
% \emph{Geom. Funct. Anal.} \textbf{8} (1998), no.~1, 123--148
% \Zbl{0903.53045} \MR{1601854}

%------ Example for a book:
% \bibitem{book1}
% W.~P. Ziemer, \emph{Weakly differentiable functions}.
% Grad. Texts in Math. 120,  Springer, New York, 1989
%\Zbl{0692.46022} \MR{1014685}

%------ Example for a paper in a book:
% \bibitem{incollection1}
% J.~S. Milne, Introduction to Shimura varieties.
% In \emph{Harmonic analysis, the trace formula, and Shimura varieties},
% edited by M.~W. Marcellin and E.~Giorgi, pp. 265--378,
% Clay Math. Proc. 4, Amer. Math. Soc., Providence, RI, 2005
% \Zbl{1148.14011} \MR{2192012}

%------ Example for a preprint on arXiv:
% \bibitem{preprint1}
% D.~V. Nguyen, S.~K. Chilappagari, M.~W. Marcellin, and B.~Vasic,
% LDPC codes from latin squares free of small trapping sets. 2010, \arxiv{1008.4177}

%------ Example for a report:
% \bibitem{report1}
% J.~Schöberl, Commuting quasi-interpolation operators.
% Technical report isc-01-10-math, Texas A\&M University, 2001,
% \url{www.isc.tamu.edu/publications-reports/tr/0110.pdf}

%------ Example for a thesis:
% \bibitem{thesis1}
% E.~Giorgi, \emph{The geometric universe}.
% Ph.D. thesis, University of Maryland, College Park, 2002

%\end{thebibliography}

\end{document}